\documentclass[12pt]{amsart}

\usepackage{amsmath,amssymb,amsthm,graphicx,mathrsfs,url, bbold, color, wrapfig}
\usepackage[open, openlevel=2, depth=3, atend]{bookmark}
\usepackage{euscript}
\hypersetup{pdfstartview=XYZ}
\usepackage{url}

\newcommand{\C}{\mathbb{C}}
\newcommand{\R}{\mathbb{R}}

\newcommand{\N}{\mathbb{N}}
\newcommand{\Z}{\mathbb{Z}}
\newcommand{\Hh}{\mathbb{H}}

\newtheorem{theorem}{Theorem}
\newtheorem*{theorem*}{Theorem}
\newtheorem{prop}{Proposition}[section]
\newtheorem{definition}[prop]{Definition}

\newtheorem{notation}{Notation}
\newtheorem{lemma}[prop]{Lemma}
\newtheorem{cor}[prop]{Corollary}
\newtheorem{remark}{Remark}

\DeclareMathOperator{\Op}{Op}
\DeclareMathOperator{\supp}{supp}
\newcommand{\lift}[1]{\mathbf{#1}}

\title{Long time Quantum evolution of observables on cusp surfaces}
\author{Yannick Bonthonneau}
\email{yannick.bonthonneau@ens.fr}
\address{DMA, U.M.R. 8553 CNRS, \'Ecole Normale Superieure, 45 rue d'Ulm,
75230 Paris cedex 05, France}
\begin{document}

\begin{abstract}
We build a semi-classical quantization procedure for finite volume manifolds with hyperbolic cusps, adapted to a geometrical class of symbols. We prove an Egorov Lemma until Ehrenfest times on such manifolds. Then we give a version of Quantum Unique Ergodicity for the Eisenstein series for values of $s$ converging slowly to the unitary axis.
\end{abstract}
\maketitle

\thispagestyle{empty}
In this paper, we work with non-compact complete manifolds $(M,g)$ of finite volume with hyperbolic ends. Such manifolds are called \emph{cusp manifolds}. They decompose into a compact manifold with boundary and a finite number of cusp-ends $Z_1, \dots Z_m$; that is, of the type: 
\begin{equation*}
Z_a^\Lambda = [a, +\infty)_y \times \mathbb{T}^d_{\theta} \text{ with the metric } ds^2 = \frac{dy^2 + d\theta^2}{y^2}
\end{equation*} 
where $ d \theta^2$ is the canonical flat metric on the $d$-dimensionnal torus $\mathbb{T}^d=\R^d/\Lambda$. The Laplacian on compactly supported smooth functions on $M$ is essentially self-adjoint, so it has a unique self-adjoint extension $\Delta$ to $L^2(M)$. Here, we take the analyst's convention that $\Delta$ is a non-positive operator. In \cite{MR699488}, Yves Colin de Verdi\`ere proved that for cusp \emph{surfaces}, the resolvent of the Laplacian has a meromorphic continuation through the continuous spectrum. Another proof was given in any dimension with a more general definition of cusps by M\"uller in \cite{MR725778}. It gives the following. The operator defined on $L^2$ for $\Re s > d/2$ and $s\notin [d/2 , d]$,
\begin{equation*}
\mathcal{R}(s) = (-\Delta - s(d-s))^{-1}
\end{equation*}
has a meromorphic continuation to the whole complex plane, as an operator $C^\infty_c(M) \to C^\infty(M)$. The poles are called resonances. Those on the vertical line $\Re s = d/2$ (called the unitary axis), and on the segment $[0, d]$ correspond to discrete $L^2$ eigenvalues. However the others are associated to continuous spectrum. The way to prove this uses a meromorphic family of eigenfunctions, the so-called \emph{Eisenstein series} $\{E_i(s)\}_{s\in \C, i=1\dots m}$. Those are smooth, \emph{not} $L^2$, and satisfy
\begin{equation*}
-\Delta E_i(s) = s(d-s) E_i(s), \quad s\in \C.
\end{equation*}
The $E(s)$ naturally replace the $L^2$ eigenfunctions as spectral data for the continuous spectrum. Actually, the data can be alternatively considered to be the values of $E$ for $s$ on the unitary axis, the full family $\{E(s)\}_{s \in \C}$, or the poles and the residues at the poles of the family --- a.k.a the resonant states. Observe that it is not always easy to translate results between those formulations. We are trying to determine whether properties of $L^2$ eigenfunctions for compact manifolds remain true for this spectral data in the non-compact case; we also seek to know if new behavior arise.

\begin{wrapfigure}{r}{0.35\textwidth}
\def\svgwidth{0.40\textwidth}
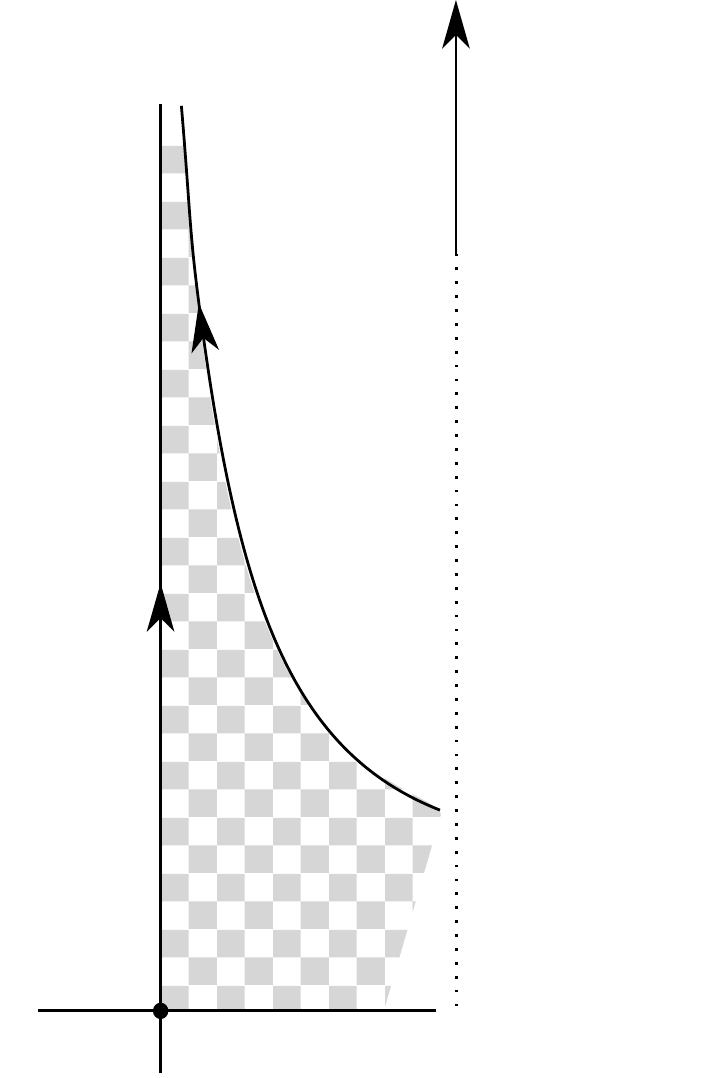
\end{wrapfigure}

The figure gives a synthetic vision of what is known (including our result) on the semi-classical measures: arrows represent asymptotics of sequences of $s$'s, and we give the associated semi-classical measure next to it.

In \cite{MR2913877}, in the case of surfaces, Semyon Dyatlov proved that when $|\Im s| \to \infty$ and $\Re s \to 1/2 + \eta$ with $\eta > 0$, the microlocal measures associated to $E_i(s)$ converge to an explicit measure $\mu_{i,\eta}$ on $S^\ast M$ --- the unit cotangent sphere. It satisfies $(X - 2\eta) \mu_{i,\eta} =0$ where $X$ is the generator of the geodesic flow. This result does not rely upon dynamical properties of the geodesic flow such as ergodicity. We recover a similar result in any dimension. When the curvature of the manifold is (strictly) negative, we adopt Babillot's argument \cite{MR1910932} which relies on the Local Product Structure to prove (see lemma \ref{lemma:convergence_measure}) that
\begin{equation*}
\shoveleft{ \mu_{i,\eta} \underset{\eta \to 0}{\rightharpoonup} \mathscr{L}_1}
\end{equation*}
where $\mathscr{L}_t$ is the normalized Liouville measure on $t S^\ast M$. In this paper, we build a semi-classical pseudo-differential calculus $\Psi(M)$ with symbols $S(M)$ and a quantization $\Op_h$ (sections 1 and 2.1). For $s\in \C$ not a resonance, let $\mu_i^h(s)$ be the distribution on $T^\ast M$
\begin{equation*}
\langle\mu_{i,j}^h(s),\sigma\rangle := \langle \Op_h(\sigma) E_i(s), E_j(s) \rangle.
\end{equation*}
Also consider
\begin{equation*}
\langle\mu^h(s),\sigma\rangle := \sum_{i} \langle \Op_h(\sigma) E_i(s), E_i(s)\rangle.
\end{equation*}

We prove:
\begin{theorem}\label{theorem:Introduction}
Assume $M$ is a cusp manifold of negative curvature. Then there is a positive constant $C_0$ such that the following holds. 
\begin{equation*}
\eta \times \mu_{i,j}^h\left(\frac{d}{2} + \eta + \frac{i}{h}\right) \underset{h \to 0, \eta \to 0}{\longrightarrow}  \delta_{ij} \pi \mathscr{L}_1
\end{equation*}
as long as 
\begin{equation}\label{eq:slow_condition}
\eta > C_0 \frac{\log |\log h|}{|\log h|}.
\end{equation}
\end{theorem}

If $s$ is a pole of $E$, then the resonant state at $s$ is a linear combination of the $E_i$'s at $1-s$. Hence, if we have a sequence of resonances $s(h)$ such that $1-s$ satisfies the hypothesis of theorem \ref{theorem:Introduction}, the semi-classical measures associated to the corresponding sequence of resonant states converges to Liouville after a suitable rescaling. However, we have no information on that rescaling; it is related to the size of the residues of the scattering determinant --- see section \ref{section:spectr_theory} for a definition --- and that seems to be a difficult problem. 

However, apart from the case of arithmetic cusp surfaces, it is quite possible that the region of the plane we are considering contains \emph{no} resonances. What is more $\eta \to \nu >0$ corresponds to a negligible part of the resonances, and we suspect it is also the case of resonances satisfying equation \eqref{eq:slow_condition}, at least for generic metrics.

It is worthwile to recall the result obtained by Steven Zelditch, in \cite{MR1105653}, of Quantum Ergodicity for values of $E(s)$ on the unitary axis, in the case of \emph{hyperbolic} surfaces. The set of poles of $\{E(s)\}_{s\in\C}$ is encoded in what is called the \emph{Scattering phase}, which is a function $\mathcal{S}$ on $\R$ (see section 2.1.1). 
\begin{theorem*}[Zelditch]
For any $T>0$,
\begin{equation*}
h \int_{-T}^T \left| \mu^h\left(\frac{d}{2} + \frac{it}{h} \right) - 2\pi \mathcal{S}'\left(\frac{t}{h}\right)\mathscr{L}_{t^2/h^2} \right| d t \underset{h \to 0}{\longrightarrow} 0
\end{equation*}
\end{theorem*}

When working on the spectrum (i.e $\Re s = d/2$), it is a difficult problem to obtain asymptotics for $\mu^h_i(d/2 + i/h)$ without averaging on the spectrum. It is an open problem how close to the spectrum one can get without some averaging; it is also not easy to make a reasonable and precise conjecture on the behaviour of those Wigner measures when $s$ gets closer to the axis without being on it. This can be contrasted with the convex cocompact case --- replacing cusps by funnels --- for which results have recently appeared --- see \cite{MR3188066} and \cite{MR3215926}.

The only surface, as far as we know, for which such results have been obtained is the \emph{modular surface}. In 2000, Luo and Sarnak proved in \cite{MR1361757} that 
\begin{equation}\label{eq:Luo_Sarnak}
\mu^h\left(\frac{1}{2} + \frac{i}{h} \right) \sim \frac{6}{\pi} |\log h| \mathscr{L}_1.
\end{equation} 
Recently \cite{2011arXiv1111.6615P}, Petridis, Raulf and Risager extended this to
\begin{equation}\label{eq:Petridis_Raulf_Risager}
\mu^h\left(\frac{1}{2} + \eta + \frac{i}{h} \right) \sim \frac{6}{\pi} |\log h| \mathscr{L}_1.
\end{equation} 
as long as $\eta |\log h| \to 0$. However, the case of the modular surface, for which most of the spectrum is discrete, is very particular. There is no reason to expect that \eqref{eq:Luo_Sarnak} and \eqref{eq:Petridis_Raulf_Risager} should hold in general for other surfaces. \\

The main tool that enables us to improve Dyatlov's theorem is a long time Egorov lemma. Let $p$ be the principal symbol of the Laplacian, that is, the metric $p(x,\xi) = |\xi|^2_x$.
\begin{theorem}
Let $\sigma \in S$ be supported in a set where $p$ is bounded. There is a symbol $\sigma_t$ in an exotic class such that for $|t| \leq C |\log h|$
\begin{equation*}
e^{-ith \Delta}\Op(\sigma) e^{ith\Delta} = \Op(\sigma_t) + \mathcal{O}( h^\infty).
\end{equation*}
This holds as long as $C>0$ remains smaller than an explicit constant $C_{max}$.
\end{theorem} 
As in most situations, the proof of our Egorov lemma is relatively easy once the settings, and in particular, the relevant properties of the quantization have been established. Let us explain; apart from some abstract nonsense, the analysis in Egorov lemma is contained in an estimate of the derivatives of the flow. In the compact case, the choice of how those derivatives are estimated is not important. However, in a non-compact case, one has to use norms consistent with the geometry of the problem, and those norms determine the class of symbols to use.

One way to avoid those problems is to use a compactly supported quantization and operators. This is adapted in cases where the whole interesting part of the underlying dynamics takes place in a compact set (see for example \cite{MR3215926}). However in our case, a crucial part of the dynamics happens in the cusps, so we wanted to allow for symbols supported in the cusps, and we had to give a treatment for the ends.

The \emph{cusp}-calculus of Melrose sees the cusp as conformal to a Euclidean cylinder (see \cite{MR1734130}), and the corresponding class of symbols has a nice behavior with respect to the euclidean dynamics on a half cylinder with boundary, considering the point at infinity of a cusp as a circle boundary. Another description we found was in \cite{2014arXiv1405.2126B}. For symbols only depending on the vertical variables in the cusps, this is probably easier to manipulate. However it does not allow for flexible behavior on the $\theta$ variable. Let us observe however that \cite{2014arXiv1405.2126B} deals with cusps that are much more general than the ones we consider.

The only closely related description of pseudo-differential operators we found was \cite{MR850155} by Zelditch. In this paper, he develops a quantization procedure for the hyperbolic plane, using Fourier-Helgason transform on the unit disk. Then, he shows that those operators can be symmetrized to operate on compact hyperbolic surfaces. In some sense, we use \emph{the} class of symbols that are compatible with the metric, and that class is similar to that introduced by Zelditch. However, the quantization procedure is different: we use only euclidean Fourier transform to build operators specifically on cusps --- see \eqref{eq:Weyl-quantization-Rd}. We prove composition stability without any proper support assumption. This is available in any dimension $\geq 2$, in a semi-classical formulation. We also state usual theorems on pseudors, including $L^2$ bounds, and a trace formula. We detailed the proofs, with elementary tools, only referring to \cite{MR2952218}.

This will be part of the author's PhD thesis.

\textbf{Acknowledgement}. We thank Nalini Anantharaman and Colin Guillarmou for their fruitful and extensive advice. We thank Semyon Dyatlov for his suggestions. We thank Barbara Schapira for her explanations on dynamical matters.

\section{Quantizing in a full cusp}

\subsection{Symbols}

Let $Z_\Lambda$ be a full cusp. That is
\begin{equation*}
Z_{\Lambda} = \R^+ \times \R^d/\Lambda,
\end{equation*}
where $\Lambda$ is some lattice in $\R^d$. The first coordinate, $y$ is called the \emph{height} ; the second is denoted by $\theta$, and we write $x=(y,\theta) = (y, \theta_1, \dots, \theta_d)$. The cusp is endowed with a cusp metric :
\begin{equation*}
 d s^2 = \frac{ d y^2 +  d \theta^2}{y^2},
\end{equation*}
where $ d \theta^2 = d\theta_1^2 + \dots + d \theta_d^2$ is the canonical flat metric on $\R^d$. By rescaling, we see that $Z_{\Lambda}$ and $Z_{t\Lambda}$ are isometric whenever $t>0$, so we assume that $\Lambda$ has covolume $1$. In the first part of the article, we will write just $Z$ for $Z_\Lambda$ because $\Lambda$ will not change. The Riemannian volume in $Z$ is 
\begin{equation*}
d \mathrm{vol}(x) = \frac{ d y  d \theta^d}{y^{d+1}}.
\end{equation*}
We refer to the space of square integrable functions with respect to this volume as $L^2(Z)$. The Laplacian is
\begin{equation*}
\Delta = y^2 \Delta_{\mathrm{eucl}} - (d-1) y \frac{\partial}{\partial y},
\end{equation*}
where $\Delta_{\mathrm{eucl}}$ is the Laplacian for the Euclidean cylinder. On the cotangent bundle $T^\ast Z$, we let $Y$ and $J$ be the dual coordinates to $\partial_y$ and $\partial_\theta$, with $\xi = Ydy + Jd\theta$ ($J$ is a vector in $\R^d$). We also write $\xi= (Y, J)=(Y, J_1, \dots, J_d)$. The usual Poisson Bracket on $T^\ast Z$ writes
\begin{equation*}
\{ f, g \} = \partial_Y f \partial_y g - \partial_y f \partial_Y g + \nabla_J f . \nabla_\theta g - \nabla_\theta f . \nabla_J g \text{ with $\nabla$ the usual flat connexion on $\R^n$.}
\end{equation*}
The riemmannian metric on $Z$ gives an isomorphism between $T Z$ and $T^\ast Z$, and $T^\ast Z$ is thus endowed with the metric
\begin{equation*}
| \xi |^2_x =  y^2 (Y^2 + |J|^2).
\end{equation*}

In appendix \ref{appendix:functionnal_spaces_cusp}, we recall how to define the spaces $\mathscr{C}^n(Z)$ of functions using covariant derivatives. This definition is intrinsic of the metric, however it is not very practical for computations. Let 
\begin{equation*}
X_y := y\partial_y \quad X_{\theta_i} := y \partial_{\theta_i}.
\end{equation*}
If $\alpha= (\alpha_1, \dots, \alpha_k)$ is a sequence with $\alpha_i \in \{ y, \theta_1, \dots, \theta_d\}$ --- we say a \emph{space-index of length k} --- we denote $X_{\alpha_1} \dots X_{\alpha_k}$ by $X_\alpha$, and $|\alpha|=k$ is the length. Then we prove in \eqref{eq:equivalent_norms_C^n} that 
\begin{equation*}
\| . \|_{\mathscr{C}^n(Z)} \text{ and }\sum_{|\alpha| \leq n} \| X_\alpha f \|_{L^\infty(Z)} \text{ are equivalent norms.}
\end{equation*}
Since $[X_y, X_\theta] = X_\theta$, the order in which the derivatives are taken has little importance.

\begin{definition}
We call \emph{coefficients} the elements of $\mathscr{C}^\infty_b(Z)= \cap_{n\geq 0} \mathscr{C}^n(Z)$, that is, smooth functions $f$ on $Z$ such that for any space-index $\alpha$, 
\begin{equation*}
\| X_{\alpha} f\|_{L^\infty (Z)} < \infty.
\end{equation*}
Elements of  $\mathscr{C}^\infty_{b,h}(Z) := C^\infty(\R^+ \to \mathscr{C}^\infty_b(Z))$ are called \emph{(semi-classical)} coefficients. 
\end{definition}

Both $\mathscr{C}^\infty_b(Z)$ and $\mathscr{C}^\infty_{b,h}(Z)$ have a natural ring structure.
\begin{lemma}
The module generated by $X_y$ and $X_\theta$'s over the coefficients $\mathscr{C}^\infty_b(Z)$ is a Lie algebra.
\end{lemma}

\begin{proof}
It suffices to consider the behavior of $[ a X_y, b X_{\theta_i}]$ with $a$, $b$ in $\mathscr{C}^\infty_b(Z)$ :
\begin{equation*}
[ a X_y, b X_{\theta_i}] = (ab + a X_y b) X_{\theta_i} - (b X_{\theta_i} a) X_y
\end{equation*}
\end{proof}

If we consider $hX_y$ and $hX_{\theta_i}$'s as vector fields on $Z$ with a parameter $h\geq 0$, we deduce that the module they generate over $\mathscr{C}^\infty_{b,h}(Z)$ is also a Lie algebra. Hence it makes sense to speak of its universal envelopping algebra $\mathcal{V}(Z)$. This is an algebra of \emph{semi-classical differential operators} on $Z$. From now on, all differential operators we use will be in $\mathcal{V}(Z)$. 

Inside of $\mathcal{V}(Z)$, we can consider the subalgebra generated by $hX_y$ and $hX_{\theta_i}$'s over $\C$; those are the \emph{constant coefficients} differential operators.

\begin{prop}
The elements of $\mathcal{V}(Z)$ can be decomposed as finite sums of the type
\begin{equation*}
\sum_{i\geq 0, \alpha}  h^{|\alpha|+i} a_{i,\alpha}(h,x) X_\alpha
\end{equation*}
with $a_{i,\alpha}$'s in $\mathscr{C}^\infty_{b,h}(Z)$.
\end{prop}

We define :
\begin{equation*}
P = -\frac{h^2}{2}\Delta.
\end{equation*}
It is in $\mathcal{V}(Z)$ and in some sense, $\mathcal{V}(Z)$ has been taylored to satisfy this property. Actually, $P$ is a constant coefficient operator: 
\begin{equation*}
\Delta = X_y^2 - d \times X_y + X_{\theta_1}^2 + \dots + X_{\theta_d}^2.
\end{equation*}

Using the algebraic properties described above, one can prove:
\begin{prop}\label{prop:semi-classical-symbol_differential_operator}
Let $A\in \mathcal{V}(Z)$, and $(x,\xi)\in T^\ast Z$. Let $\phi$ be a smooth function on $Z$, with $\phi(x) = 0$ and $ d  \phi(x) = \xi$. Then we let
\begin{equation*}
\sigma^0(A)(x,\xi) =\lim_{h\to 0} A_h(e^{ i  \phi /h}) (x).
\end{equation*}
This limit exists and does not depend on the choice of $\phi$. It is called the \emph{principal symbol} of $A$. It is a polynomial in the $\xi$ variable, smoothly depending on $x$ and $h$. What is more, its monomials are of the form $a(x)y^{k+l} Y^k J^l$ where $a\in C^\infty_b(Z)$.

The mapping $\sigma^0$ from $\mathcal{V}(Z)$ to functions on $T^\ast Z$ is linear. What is more, if $A, B$ are in $\mathcal{V}(Z)$, 
\begin{gather*}
\sigma^0(AB) = \sigma^0(A) \sigma^0(B) \quad \text{ and } \quad \sigma^0 \left( \frac{1}{h} [ A, B] \right) =  \left\{\sigma^0(A), \sigma^0(B) \right\} \\
\sigma^0 \left(\frac{h}{i}X_y \right)= yY  \quad \text{ and } \quad \sigma^0\left( \frac{h}{i} X_{\theta_i} \right) = y J_i
\end{gather*}
\end{prop}
We deduce that the principal symbol of $P$ is the metric 
\begin{equation}\label{eq:symbol_laplacian}
p=|\xi |^2_x/2.
\end{equation}

If we consider the set of functions $S_{\mathcal{V}}$ on $T^\ast Z$ obtained by taking principal symbols of differential operators, we see that it is stable by the action of the vectors $X_y$ and $X_{\theta_i}$'s. In the $\xi$ direction, we introduce
\begin{equation*}
X_Y := \frac{1}{y} \partial_Y \quad X_{J_i} := \frac{1}{y}\partial_{J_i}.
\end{equation*}
Then $X_Y$ and $X_{J_i}$ also stabilize $S_\mathcal{V}$. Let us refine this description. $S_\mathcal{V}$ is a graded algebra decomposing as $S_\mathcal{V} = \cup_{n\geq -\infty} S^n_\mathcal{V}$, where $S^n_\mathcal{V}$ is the set of $q$'s of degree lesser than $n$. Then $X_y$ and $X_\theta$'s map $S^n_\mathcal{V}$ to itself while $X_Y$ and $X_J$'s map $S^n_{\mathcal{V}}$ to $S^{n-1}_\mathcal{V}$. Also remark that whenever $q\in S^n_\mathcal{V}$, $q = \mathcal{O}(\langle y \xi \rangle^n)$ where $\langle  x\rangle$ is the usual bracket $\sqrt{1+ x^2}$.

The above motivates the introduction of the following class of symbols:
\begin{definition}
Let $\sigma$ be a smooth function on $T^\ast Z$. We say that $\sigma$ is a \emph{(hyperbolic) symbol} on $Z$ of order $n \in \R$ if for any finite sequence $\{ \alpha_k \}$ with $\alpha_k \in \{ y,\theta_{1\dots d}, Y, J_{1\dots d}\}$, if $(\alpha)$ is the number of $Y$'s and $J$'s in the sequence,
\begin{equation}\label{eq:symbolic-estimate}
q_{n,\alpha} := \sup_{T^\ast X_i} | X_{\alpha_1} \dots X_{\alpha_k} a(x, \xi) | \leq C \langle y\xi \rangle^{n-(\alpha)}.
\end{equation}

We let $S(Z)$ be the set of symbols, $S^n(Z)$ be the set of symbols of order $n$, and $S^{-\infty} = \cap_{n\in \R} S^n(Z)$. $S$ is graded by the order, and
\begin{align*}
S^n_\mathcal{V} \subset & S^n(Z) \\
X_y, X_{\theta_i} : & S^n(Z) \to S^n(Z) \\
X_Y, X_{J_i} : & S^n(Z) \to S^{n-1}(Z) \\
\sigma \in S^m, \mu \in S^n & \Rightarrow \{ \sigma, \mu \} \in S^{m+n-1}
\end{align*}

The family of semi-norms $q_{n, \alpha}$ gives a structure of metric space to $S^n$.
\end{definition}

We have not specified an $h$-dependency. Actually, we will need to let symbols depend on $h$ in a slightly rough fashion, so we use the classes :
\begin{definition}
Let $0 \leq \rho < 1/2$. Consider complex functions $\sigma$ of $h>0$ and $(x,\xi)\in T^\ast Z$ such that for fixed $h$, $\sigma_h$ is in $S^n$. If $\alpha$ is a sequence of indices, let $|\alpha|$ be its length. Assume that $\sigma$ additionnally satisfy the family of estimates
\begin{equation*}
q_{n,\alpha,\rho} := \sup_{h} h^{\rho |\alpha|} q_{n, \alpha}(\sigma_h) < \infty, \text{ for $\alpha$ finite sequence of }y,\theta_i,Y,J_i.
\end{equation*}
Then we say that $\sigma$ is an exotic symbol of order $(n,\rho)$, and write $\sigma \in S^n_\rho(Z)$. We also define $S_\rho (Z) = \cup S^n_\rho (Z)$ and $S^{-\infty}_\rho (Z) = \cap S^n_\rho (Z)$. We have:
\begin{align*}
S^n_\mathcal{V} \subset & S^n(Z) \\
h^\rho X_y, h^\rho X_{\theta_i} : & S^n(Z) \to S^n(Z) \\
h^\rho X_Y, h^\rho X_{J_i} : & S^n(Z) \to S^{n-1}(Z) \\
\sigma \in S^m, \mu \in S^n & \Rightarrow h^{2\rho}\{ \sigma, \mu \} \in S^{m+n-1}
\end{align*}
The semi-norms $q_{n,\alpha,\rho}$ also give a structure of metric space to $S^n_\rho$.
\end{definition}

The rest of section 1 is devoted to describing a quantization procedure for this algebra of symbols. In section 1.2, we give a definition, and prove that we obtain pseudo-differential operators in the usual sense. Then, in section 1.3, we give a stationary phase lemma. We use it to prove usual properties of the quantization in section 1.4.

\subsection{Quantizing symbols.}

After we give a quantization procedure $\Op$ in \ref{definition:Quantization} for $\sigma \in S_\rho$, we first prove that the operators we obtain have pseudo-differential behaviour locally. That is, if $\gamma$ is a diffeomorphism from a relatively compact open set $U$ in $\R^n$ onto its image in $Z$, the pullback
\begin{equation*}
(\gamma^\ast \Op(\sigma)) f(x) = \left[\Op(\sigma)(f\circ \gamma) \right] ( \gamma(x)).
\end{equation*}
is a pseudo-differential operator in $U$. Then we prove (lemma \ref{prop:support_around_y-diagonal}) that $\Op$ is \emph{pseudo-local} in the following sense : if $\phi_1$ and $\phi_2$ are two coefficients on $Z$ not depending on $\theta$, with disjoint support, 
\begin{equation*}
\phi_1 \Op(\sigma) \phi_2 = O_{H^{-n} \to H^n}( h^\infty) \quad \text{ for every }n\in \N.
\end{equation*}
The Sobolev spaces on cusps are defined in appendix \ref{appendix:functionnal_spaces_cusp}. To get to \ref{prop:support_around_y-diagonal}, we have to prove that our operators have \emph{some} Sobolev regularity --- see proposition \ref{prop:Sobolev_boundedness}. This is deduced, with a usual parametrix argument, of the crucial lemma \ref{lemma:roughL2boundedness} about regularity on $L^2(Z)$. The regularity we prove in this section is certainly not optimal, and we will get better statements --- see proposition \ref{prop:optimal_L^2_regularity} --- once we obtain stability by composition.

For convenience, we will use some expressions in the half-space $\Hh^{d+1} = \R^+\times\R^d$, which is the universal cover of $Z$. If $a$ is some function on $Z$ (resp. $T^\ast Z$), we identify it with its unique lift to $\Hh^{d+1}$ (resp. $T^\ast \Hh^{d+1}$). We denote by $\mathbf{Op}_h^w$ the \emph{usual Weyl quantization} of a symbol on $\R^{d+1}$ :
\begin{equation}\label{eq:Weyl-quantization-Rd}
\mathbf{Op}_h^w(\eta)f = \frac{1}{(2\pi h)^{d+1}}\int_{\R^{d+1}\times\R^{d+1}}e^{ i  \langle x-x',\xi\rangle/h} \eta\left( \frac{x+x'}{2}, \xi\right) f(x')  d x'  d \xi.
\end{equation}
If $\eta$ and $f$ are only defined on the upper half space, this also makes sense (continuing all functions by $0$ in the lower half space). Now, assume that $\eta$ is actually a symbol of order $n$, and $f$ is a smooth compactly supported function on $Z$ --- that is a periodic function in $\theta$ in $\Hh^{d+1}$, with compact support in $y$. Then, in each fiber of $T^\ast Z$, $\eta$ is a symbol in the usual sense of $\R^n$, with uniform estimates as long as $y$ stays in a compact set of $\R^{+\ast}$. We deduce that the Fourier transform of $\eta$ fiberwise is a distribution in $\mathscr{S}'$, whose singular support is $\{0\}$ and with fast decay at infinity, with estimates uniform in $x$ as long as $y$ stays in a compact set. We deduce that $\mathbf{Op}_h^w(\eta)f$ is well defined. Actually, applying a finite number of $X_y$'s and $X_\theta$'s, we can repeat the argument and observe it is also a smooth function. A similar argument will be detailed in the proof of lemma \ref{lemma:roughL2boundedness}. Now, we can give our basic definition :
\begin{definition}\label{definition:Quantization}
Let $\sigma\in S(Z)$ be some hyperbolic symbol on $T^\ast Z$. Then, for any $f\in C^\infty_c (Z)$, we let
\begin{equation*}
\Op_h(\sigma)f(y,\theta) = y^{\frac{d+1}{2}} \mathbf{Op}_h^w(\lift{\sigma}) \left(\frac{1}{y^{\frac{d+1}{2}}}f\right).
\end{equation*}
This is seen to be a $\Lambda$-periodic function in the $\theta$ direction. $\Op_h(\sigma)$ defines an operator $C^\infty_c(Z) \to C^\infty(Z)$. When $\sigma$ is a symbol, $\Op(\sigma)$ denotes the family $\{ \Op_h(\sigma)\}_h$. 

We let $\Psi_\rho^n$ be the set of $\{ \Op(\sigma) | \sigma \in S_\rho^n\}$.
\end{definition}

The introduction of $(y/y')^{(d+1)/2}$ corresponds the conjugacy by the unitary map $f \to y^{(d+1)/2} f$ from $L^2(dy d\theta)$ to $L^2(Z)$.

\subsubsection{The principal symbol}

In this section, we check that:

\begin{lemma}\label{lemma:Pseudo-differential_behavior}
The definition of the semi-classical principal symbol given for elements of $\mathcal{V}(Z)$ in \ref{prop:semi-classical-symbol_differential_operator} extends to operators $\Op(\sigma)$'s, and $\sigma^0(\Op(\sigma))=\sigma$.
\end{lemma}

\begin{proof}
Let $\chi$ be some $C^\infty_c(Z)$ function equal to $1$ near $x\in Z$. Let $\xi_0 \in T^\ast_x Z$ and $\phi$ some smooth function on $Z$ such that $\phi(x)= 0$ and $ d \phi(x) = \xi_0$. Let $\sigma \in S_\rho$, take $\tilde{x}$ a lift of $x$ in $\Hh^{d+1}$ and integrating over $\Hh^{d+1}\times\R^{d+1}$
\begin{equation*}
\Op_h(\sigma)\left(\chi e^{ i  \phi/h}\right)(x) =\frac{1}{(2\pi h)^{d+1}} \int  e^{ i  (\langle \tilde{x}-x',\xi\rangle + \phi(x'))/h} \left(\frac{y}{y'}\right)^{(d+1)/2}\sigma\left(\frac{\tilde{x}+x'}{2}, \xi\right) \chi(x') d x' d \xi
\end{equation*}
The proof here will be very similar to the case of $\R^n$, so we just insist on the differences. We let $\tilde{\chi}$ be some smooth compactly supported function equal to $1$ around $\tilde{x}$ ; we insert $1=\tilde{\chi}+(1-\tilde{\chi})$ to break the integral into two parts (I) and (II).

For the first term (I), observe that it is an integral over a fixed compact set in the $x'$ variable, with an integrand that has symbolic behavior in the $\xi$ variable, and a very simple phase function. Classical stationary phase results directly apply to give that :
\begin{equation*}
\lim_{h\to 0} \mathrm{(I)} = \sigma(x,\xi_0).
\end{equation*}

To estimate (II), we integrate by parts in $\xi$; we just have to introduce suitable powers of $(y+y')$ to obtain the new integrand 
\begin{equation*}
C_k h^{2k-d-1} e^{ i  (\langle \tilde{x}-x',\xi\rangle + \phi(x'))/h} \frac{(y+y')^{2k}(1-\tilde{\chi}(x'))}{|\tilde{x}-x'|^{2k}} \left(\frac{y}{y'}\right)^{(d+1)/2} \sigma^\ast_k\left(\frac{\tilde{x}+x'}{2}, \xi\right) \chi(x'),
\end{equation*}
where $\sigma^\ast_k = (X_Y^2 + X_J^2)^k \sigma$ and $C_k=(i/2)^{2k-d-1} \pi^{-d-1}$. With $k$ big enough, this is $\mathcal{O}(h^{(1-\rho)2k-d-1})$ in $L^1( d x' d \xi)$.

\end{proof}

\subsubsection{Basic boundedness estimates.}

\begin{lemma}\label{lemma:roughL2boundedness}
For all $\epsilon >0$, the elements of $\Psi_\rho^{-d-1-\epsilon}$ extend to bounded operators on $L^2(Z)$, with $\mathcal{O}(h^{-\rho d-1})$ norm as $h\to 0$.
\end{lemma} 

In the subsequent developments, we will call \emph{Schwarz Kernel} of $A: C^\infty_c \to C^\infty$ the distribution defined by :
\begin{equation*}
Af(x) = \int_Z K(x,x') f(x')  d x'
\end{equation*}
where $ d x' = dy d \theta$ is the Lebesgue measure in the half-cylinder. Let us state a modified Schur lemma --- see page 82 in \cite{MR2952218} for the original version.
\begin{lemma}\label{lemma:Schur_boundedness} 
Let $A$ be an operator from $C^\infty_c(Z)$ to $C^\infty(Z)$, with Schwarz kernel $K$. Assume that for some $\tau \in \R$,
\begin{equation}\label{eq:Schur_estimate}
C(A,\tau) := \sup_{x} \int_{x'\in Z} \left(\frac{y'}{y}\right)^{d+1+\tau} |K(x,x')|dx' \times \sup_{x'} \int_{x \in Z} \left(\frac{y}{y'}\right)^\tau |K(x,x')| dx < \infty.
\end{equation}
Then $A$ can be extended to a bounded operator on $L^2$ with
\begin{equation*}
\| A \|_{L^2 \to L^2 }^2 \leq C(A,\tau).
\end{equation*}
\end{lemma}
 
\begin{proof}
We follow the classical proof. All the integrals are over $Z$. if $u\in C^\infty_c(Z)$,
\begin{align*}
| Au(x)| & \leq \int \sqrt{y'^{d+1 + \tau}|K|} \sqrt{y'^{-d-1 - \tau} |K| |u|^2}  d x' \\
		& \leq \| y'^{d+1+ \tau} |K| \|_{L^1(x')}^{1/2}\left[ \int y'^{-\tau}|K(x,x')| |u|^2 y'^{-d-1} d x' \right]^{1/2} \\
\int |Au(x)|^2 y^{-d-1}  d x &\leq \left[ \sup_{x} \int_{x'} \left(\frac{y'}{y}\right)^{d+1+\tau} |K(x,x')| \right] \int_{x,x'} \left( \frac{y}{y'}\right)^\tau |K(x,x')| \frac{|u|^2}{y'^{d+1}} \\
									&\leq C(A,\tau)\int \frac{|u|^2 dx'}{y'^{d+1}}
\end{align*}
\end{proof}

Now, we prove lemma \ref{lemma:roughL2boundedness}
\begin{proof}
Let $\sigma \in S^{-d-1-\epsilon}_\rho$ with some $\epsilon>0$. Formula \eqref{eq:Weyl-quantization-Rd} actually defines $\Op^w(\sigma)$ acting on the half plane. We let $K^w_\sigma(x,x')$ be its kernel, and we let $K_\sigma$ be the kernel of $\Op_h(\sigma)$ --- $K_\sigma$ depends on $h$. Then we have
\begin{equation}\label{eq:summing_kernel_quotient}
K_\sigma (y,\theta, y', \theta') = \left(\frac{y}{y'}\right)^{\frac{d+1}{2}}\sum_{k\in \Lambda} K^w_\sigma(y,\theta, y',\theta' + k).
\end{equation}
Plugging this identity in \eqref{eq:Schur_estimate}, we see that instead of integrating over the cusp $Z$, we can integrate over the half space; this way, we prove that $C(\Op(\sigma), \tau)$ is less than
\begin{equation*}
\left[\sup_x \int_{x'\in \Hh^{d+1}} \left( \frac{y'}{y}\right)^{\frac{d+1}{2} + \epsilon} |K^w_\sigma(x,x')| dx'\right] \left[\sup_{x'} \int_{x \in \Hh^{d+1}} \left( \frac{y}{y'}\right)^{\frac{d+1}{2} + \epsilon} |K^w_\sigma(x,x')| dx\right] .
\end{equation*}
By linearity of $\Op$, it suffices to consider real symbols, so we assume that $\sigma$ is real. Then $K^w_\sigma(x,x') = \overline{K^w_\sigma(x',x)}$. By symmetry of the two terms in the above equation, it suffices to prove that for some $\tau \in \R$, the first term is finite.

Since $\sigma$ is of order strictly less than $-d-1$, it is integrable in the fibers, and we can estimate : 
\begin{equation*}
|K^w_\sigma(x,x')| \leq  \frac{1}{(2\pi h)^{d+1}} \int_{\R^{d+1}} d \xi \left|\sigma(\frac{x+x'}{2}, \xi) \right| \leq C_\epsilon \frac{1}{h^{d+1}} \frac{1}{(y+y')^{d+1}}q_{d+1+\epsilon, 0}(\sigma).
\end{equation*}

With no decay in $\theta$, this is obviously not sufficient to prove boundedness. We observe the following fact :
\begin{equation}\label{eq:Differentiating-J}
\frac{\theta - \theta'}{ i  h \frac{y+y'}{2}} K^w_\sigma = K^w_{X_J \sigma}.
\end{equation}
Using \eqref{eq:Differentiating-J} $d+1$ times, and the definition of symbols, we get that :
\begin{equation}\label{eq:growth_away_diagonal}
|K^w_\sigma(x,x')| \leq C \frac{1}{h^{d+1}} \frac{1}{(y+y')^{d+1}} \frac{1}{1+ h^{(d+1)\rho}\left| \frac{\theta - \theta'}{h (y+y')}\right|^{d+1}}.
\end{equation}
Now, we integrate \eqref{eq:growth_away_diagonal} in the $\theta'$ variable. Actually, we translate by $\theta$, and we rescale with $\mu = h^{\rho - 1} (\theta' - \theta)/(y+y')$ to get:
\begin{equation*}
\int |K(y,\theta, y', \theta')| d \theta' \leq C h^{-\rho d - 1} \frac{1}{y+y'} \int_{\R^d} \mathrm{d}\mu \frac{1}{1 + |\mu|^{d+1}} \leq \frac{C}{h^{1+\rho d}} \frac{1}{y+y'} .
\end{equation*}
We just have to find $\tau$ such that
\begin{equation*}
\sup_{y>0} \int_{y'>0} \frac{1}{y+y'} \left(\frac{y'}{y} \right)^{\frac{d+1}{2} + \tau}  d y' < \infty.
\end{equation*}
and then the norm of $\Op(\sigma)$ on $L^2$ will be $\mathcal{O}(h^{-\rho d -1})$ times some symbol norm. Changing variables to $u=y'/y$, the LHS is
\begin{equation*}
\sup_{y> 0} \int_{0}^\infty \frac{u^{\frac{d+1}{2}+\tau}}{1+u}   d u = \int_0^\infty \frac{u^{\frac{d+1}{2}+\tau}}{1+u}   d u.
\end{equation*}
For $\tau \in ]-(d+1)/2 - 1, -(d+1)/2[$, this is a convergent integral.
\end{proof}

If we wanted an optimal statement in terms of regularity, we could remark here that we only use $d+1$ symbol estimates (differentiating only in the $J$ direction) to obtain this result.

\subsubsection{Sobolev regularity}

Recall that all functionnal spaces are defined in appendix \ref{appendix:functionnal_spaces_cusp}.

We need a parametrix lemma for the composition of a pseudor with a differential operator:
\begin{lemma}\label{lemma:parametrix_1}
Let $\sigma\in S^n_\rho$, $Q_{1,2,}$ be constant-coefficient elements of $\mathcal{V}(Z)$, of order $k_{1,2}$. Then there exists a symbol $\tilde{\sigma} \in S^{n+k_1+k_2}_\rho$ such that
\begin{align*}
Q_1 \Op(\sigma) Q_2 &= \Op(\tilde{\sigma}), \\
\tilde{\sigma} 		&= \sigma \times \sigma^0(Q_1) \sigma^0(Q_2) + \mathcal{O}(h^{1-\rho} \Psi^{n+k_1+k_2-1}_\rho).
\end{align*}

Additionally, let $Q_{3,4}$ also be constant-coefficient differential operators with order $k_{3,4}$, and $N \in \N$, satisfying the ellipticity condition that $\sigma^0(Q_3) \sigma^0(Q_4)$ does not vanish. Then, there exists a symbol $\tilde{\sigma}_N$ of order $n+ k_1 + k_2 - k_3 - k_4$, such that :
\begin{align}
Q_1 \Op(\sigma) Q_2 &= Q_3\Op(\tilde{\sigma}_N) Q_4 + \mathcal{O}(h^N \Psi^{-N}_\rho), \label{eq:parametrix_1} \\
\tilde{\sigma}_N 	&= \sigma \frac{\sigma^0(Q_1) \sigma^0(Q_2)}{\sigma^0(Q_3) \sigma^0(Q_4)} + \mathcal{O}(h^{1-\rho} \Psi^{n+k_1+k_2 - k_3 - k_4 -1}_\rho).\notag
\end{align}
\end{lemma}

\begin{proof}
We start with the first part. Proceeding by induction, we see that it is enough to prove the property when $Q_{1,2}$ are constants, or first order differential operators. The case of constants is straightforward. Now, let us assume $Q_1 = hX_\theta$ and $Q_2 =1$. The kernel of $Q_1\Op_h(\sigma)$ is
\begin{align*}
{} &  \frac{1}{(2\pi h)^{d+1}}\int_{\R^{d+1}} e^{i\langle x- x',\xi\rangle} \left(\frac{y}{y'}\right)^{\frac{d+1}{2}}\left[ i  y J \sigma\left(\frac{x+x'}{2}, \xi\right) + \frac{h}{2} y \partial_\theta\sigma\left(\frac{x+x'}{2},\xi\right) \right]d\xi\\
\intertext{Decomposing $y= (y+y')/2 + (y-y')/2$, integrating by part in the $\xi$ variable when necessary, we get}
	& = \frac{1}{(2\pi h)^{d+1}}\int_{\R^{d+1}} e^{i\langle x- x',\xi\rangle} \left(\frac{y}{y'}\right)^{\frac{d+1}{2}}\left[ \sigma_\theta\left(\frac{x+x'}{2}, \xi\right) \right]d\xi \\
\intertext{where}
\sigma_\theta	& =  i  yJ\sigma - \frac{h}{2} yJ X_Y \sigma + \frac{h}{2} X_\theta\sigma + i\frac{h^2}{4} X_\theta X_Y \sigma.\\
\sigma_\theta	& =  i  yJ\sigma  + \mathcal{O}(h^{1-\rho} S^n_\rho). \\
\intertext{similarly, for $Q_1= hX_y$, we get $Q_1 \Op(\sigma) = \Op(\sigma_y)$ with}
\sigma_y		&= i  yY\sigma + \frac{h}{2} ((d+1)\sigma - yY X_Y \sigma) + \frac{h}{2} X_y\sigma + i\frac{h^2}{4} X_Y X_y  \sigma.
\end{align*}
The case when $Q_1=1$ and $Q_2 = hX_\theta, hX_y$ will lead to similar computations, and the same conclusion.

Now, we prove the second part of the lemma. We look for a semi-classical expansion for $\tilde{\sigma}_N$, in the following form : $\tilde{\sigma}_N = \sum_0^{\infty} h^{(1-\rho)k}\sigma_k$ with $\sigma_k \in S^{(n+k_1+k_2 - k_3 - k_4) - k}$. Injecting this formal development in \eqref{eq:parametrix_1}, we find a linear system of equations on the $\sigma_k$'s. Actually, identifying powers of $h$, we see that this system is in lower-triangular form. The diagonal coefficients are all the same, equal to $\frac{\sigma^0(Q_1) \sigma^0(Q_2)}{\sigma^0(Q_3) \sigma^0(Q_4)}$. The ellipticity condition is sufficient to see that this system has a unique solution of the above form. 

Our formal series does not converge, so we truncate at order $M$ for some integer $M\gg 1$. The remainder is then $\mathcal{O}(h^{(M+1)(1-\rho)}\Psi_\rho^{n+k_1+k_2 - k_3 - k_4 -M -1})$. This is $\mathcal{O}(h^N \Psi_\rho^{-N})$ for $M$ big enough; we take $\tilde{\sigma}_N$ to be this truncated series.
\end{proof}

\begin{prop}\label{prop:Sobolev_boundedness}
For all $\epsilon>0$, the elements of $\Psi^{n-d-1-\epsilon}$ are bounded from $H^s$ to $H^{s-n}$ for all $s, n$ real numbers, with norm $\mathcal{O}(h^{-|s|-|n|-\rho d -1})$.
\end{prop}

\begin{proof}
Proceeding by interpolation, we only need to prove this result for $s,n$ even integers. Let $\sigma \in S^k_\rho$ with $k < n- d- 1$. Then by \eqref{eq:equivalent-norms}
\begin{equation*}
\| \Op_h(\sigma) \|_{H^s \to H^{s-n}} =  h^{-|s|-|n|}\| (P + 1)^{(s-n)/2}\Op_h(\sigma)(P + 1)^{-s/2} \|_{L^2 \to L^2}.
\end{equation*}
By lemma \ref{lemma:parametrix_1}, there is a symbol $\tilde{\sigma}_N \in S^{k-n}_\rho$ such that
\begin{equation*}
(P + 1)^{(s-n)/2}\Op_h(\sigma)(P + 1)^{-s/2} = \Op_h(\tilde{\sigma}_N) + (P + 1)^{-(s-n)^-/2}\left[ \mathcal{O}(h^N \Psi_\rho^{-N})\right](P + 1)^{-s^+/2}.
\end{equation*}
Now, we only have to apply lemma \ref{lemma:roughL2boundedness} to $\tilde{\sigma}_N$ to conclude since $(P+1)^{-k}$ is bounded on $L^2$ as soon as $k\geq 0$.
\end{proof}

\subsubsection{Pseudo-locality statements}

Before going on to prove pseudo-locality, we need to define what we mean when we say that a family of operators is smoothing.

\begin{definition}\label{def:smoothing_cusp}
We say that a family of operators $\{A_h\}_{h>0}$ on $L^2(Z)$ is \emph{smoothing} if for every $h>0$ and $n>0$, $A_h$ maps $H^{-n}$ to $H^n$ in a continuous fashion. Additionnaly, we require that the following semi-norms
\begin{equation*}
\| . \|_{n,n} = \sup_{h >0} \| . \|_{H^{-n}\to H^n}, \quad n\in \N.
\end{equation*} 
are finite. We refer to the space of smoothing operators as $\Psi^{-\infty}$. The semi-norms give a topology to that space. We let $\Psi_\rho = \Psi^{-\infty} \cup_n \Psi_\rho^n$.

A family $\{A_h\}_{h>0}$ is said to be \emph{asymptotically} smoothing if for every $n>0$, there is a $h_n >0$ such that for every $0<h<h_n$, $A_h$ is uniformly bounded from $H^{-n}$ to $H^n$. This space is also endowed with semi-norms
\begin{equation*}
\| . \|_{n,n,k} = \sup_{0< h < 1/k} \| . \|_{H^{-n}\to H^n} \quad n,k\in \N.
\end{equation*}

Finally, we say that a family of operators is (asymptotically) negligible if it is $\mathcal{O}(h^\infty)$ in the space of (asymptotically) smoothing operators. The space of negligible operators is denoted $\mathcal{O}(h^{\infty})\Psi^{-\infty}$.
\end{definition}

We deduce of \ref{prop:Sobolev_boundedness}

\begin{cor}\label{cor:composition_negligible}
The composition of a negligible (resp. asymptotically negligible) operator with a pseudor is still a negligible (resp. asymptotically negligible) operator.
\end{cor}

\begin{notation}
We denote by $S(\R,\alpha)$ the class of symbols on $\R$ of order $\alpha$, meaning that $\eta \in S(\R, \alpha)$ when for all $k\geq 0$, there is a constant $C_k >0$ with
\begin{equation*}
\eta^{(k)}(u) \leq C_k \langle u \rangle^{\alpha - k}.
\end{equation*}

Let $K$ be the kernel of some operator $A$ on $C^\infty_c(Z)$. Let $\eta \in S(\R, \alpha)$. Then define $A_\eta$ to be the operator with kernel
\begin{equation*}
K_\eta (x,x') = K(x,x') \eta\left( \frac{y'}{y}- \frac{y}{y'}\right).
\end{equation*}
\end{notation}

\begin{prop}\label{prop:support_around_y-diagonal}
Let $\eta \in S(\R,\alpha)$ vanish near $0$ with $\alpha \leq 0$. Let $\sigma \in S_\rho$. Then $\Op(\sigma)_\eta$ is negligible.
\end{prop}

\begin{proof}
We first give bounds on $L^2$. Recall that $K^w_\sigma$ is the kernel of the operator in \eqref{eq:Weyl-quantization-Rd}. Similarly to \eqref{eq:Differentiating-J}, we have:
\begin{align}
\frac{y - y'}{ i  h \frac{y+y'}{2}} K^w_\sigma &= K^w_{X_Y \sigma}.\label{eq:Differentiating-Y}\\
\intertext{from this, we deduce that for all $N\in \N$, }
\Op(\sigma)_\eta &= ( i  h/2)^{N}\Op((X_Y)^N\sigma)_{\eta_N},\notag\\
\intertext{where }
\eta_N\left(t-\frac{1}{t}\right)&= \left(\frac{1+t}{1-t}\right)^{N}\eta\left(t-\frac{1}{t}\right),\notag
\end{align}
so that $\eta_N \in S(\R, \alpha)$. For $N$ big enough, $X_Y^N \sigma \in S_\rho^{-d-2}$. From lemma \ref{lemma:roughL2boundedness}, we thus determine that for all $N\geq 0 $ 
\begin{equation*}
\| \Op_h(\sigma)_\eta \|_{L^2 \to L^2} = \mathcal{O}(h^{N(1-\rho)}),
\end{equation*}
where the constant depends on $\|\eta_N \|_\infty$ --- which is finite since $\alpha \leq 0$.

Now, up to some fixed power of $h$, the $H^{-2N} \to H^{2N}$ norm is bounded by 
\begin{equation}\label{eq:proving_pseudo-local}
\| (P+1)^N \Op_h(\sigma)_\eta (P+1)^N \|_{L^2 \to L^2}.
\end{equation}
Observe that composition with $X_\theta$ commutes with the $A \to A_\eta$ operation. Further, 
\begin{equation*}
X_y \left[\eta(\frac{y}{y'}-\frac{y'}{y})\right] = \left[\frac{y}{y'}+\frac{y'}{y}\right] \eta'(\frac{y}{y'}-\frac{y'}{y})  = \eta^\ast\left( t - \frac{1}{t}\right) \text{ with }\eta^\ast \in S(\R, \alpha).
\end{equation*}
Combining this with \eqref{eq:Differentiating-Y}, we deduce that if we expand both $(P+1)^N$'s in \eqref{eq:proving_pseudo-local}, we will get a finite sum of $\Op(\sigma^*)_{\eta^*}$, with $\sigma^\ast$ in $S_\rho$, and $\eta^\ast$ in $S(\R, \alpha)$ still vanishing near $0$. We can apply the first part of our proof to conclude.
\end{proof}

\begin{remark}
Actually, if we take $\eta(u) = \tilde{\eta}(u/h^{\rho'})$, and go through the above proof, we see that it works as long as $\rho' < 1-\rho$. We deduce that the kernel of $\Op_h(\sigma)$ is essentially supported at distance $h^{1-\rho}$ of $\{y=y'\}$.
\end{remark}

\subsection{Stationary Phase}

Now that we proved that off-diagonal terms in the kernel of our pseudors give rise to negligible operators, it is legitimate to cutoff the kernels and keep only the part supported near the diagonal. While proving composition formulae, or when changing quantizations, this will produce in the equations expressions of the type
\begin{equation*}
\sigma_1(x_0, \xi_0) \times \sigma_2(x_1, \xi_1) \times \chi( x_0, x_1)
\end{equation*}
where $\chi(x_1, x_2)$ is a function of $y_1/y_2$ supported near $1$. This motivates the introduction of
\begin{definition}\label{definition:multi-variable_symbols}
For $\epsilon>0$, let $\Omega_{k,\epsilon}$ be the subset of $(T^\ast Z)^{k+1}$ :
\begin{equation*}
\Omega_{k, \epsilon} = \left\{ (x_0, \xi_0 ; x_1, \xi_1, \dots, x_k, \xi_k) \in (T^\ast Z)^{k+1} \quad | \quad \forall i, \epsilon \leq \left|\frac{y_i}{y_0}\right| \leq 1/\epsilon \right\}.
\end{equation*}
Let $\sigma$ be some smooth function $(T^\ast Z)^{k+1}\times [0,h_0[$, supported in some $\Omega_{k,\epsilon}$. We say that $\sigma$ is a $(k, \rho)$-symbol if it is a symbol w.r.t the weights
\begin{equation*}
\{ \langle y_0 \xi_0 \rangle^{\beta_0} \dots \langle y_0 \xi_k \rangle^{\beta_k} | (\beta_0, \dots, \beta_k) \in \R^{k+1} \}
\end{equation*}
and the vector fields $X_{x,i}=y_0 \partial_{x_i}$ and $X_{\xi,i}=1/y_0 \partial_{\xi_i}$, losing a constant $h^{-\rho}$ when differentiating. By this we mean that there is $\beta \in \R^{k+1}$ such that whenever $\alpha$ is a finite sequence of indices $\alpha_j \in \{ (x,i), (\xi,i) | i=0 \dots k \}$, if $n_i$ is the number of $(\xi,i)$ in the sequence, 
\begin{equation}\label{eq:semi-norms_multivariable}
|X_\alpha \sigma | \leq C_\alpha \langle y_0 \xi_0 \rangle^{\beta_0 - n_0} \dots \langle y_0 \xi_k \rangle^{\beta_k - n_k} \text{ for some constant $C_\alpha >0$.}
\end{equation}
In particular, a $(0,\rho)$-symbol is just a symbol in $S_\rho$. The semi-norms defined in \eqref{eq:semi-norms_multivariable} give a topology to the space of $(k,\rho)$-symbols.
\end{definition}

For $\sigma$ a $(k,\rho)$-symbol, we define the following function on $T^\ast Z$ :
\begin{equation*}
T_k \sigma : (x, \xi, h) \mapsto (2\pi h)^{-k (d+1)}\int_{i=1\dots k}  e^{\frac{ i }{h} [\sum \langle x_i-x,\xi_i - \xi \rangle ]} \sigma(x, \xi ; (x_1,\xi_1), \dots, (x_k,\xi_k))  d x_i d \xi_i,
\end{equation*}
where the integration has been taken over the universal cover $T^\ast(\Hh^{d+1})^k$. 
\begin{remark}
$T_k \sigma$ is well defined. Indeed, if we first perform the integration in the $\xi_i$ variables, we obtain Fourier transforms of symbols. Those are distributions whose singular support is reduced to $\{0\}$ and are decreasing faster than any power at infinity. Using compact support in $y_{1, \dots, k}$ --- depending on $y_0$ --- we see that such distributions can be integrated against $1$.
\end{remark}

Recall the notation
\begin{equation*}
\nabla_x.\nabla_\xi  = \partial_y \partial_Y + \sum_{i=1\dots d} \partial_{\theta_i}\partial_{J_i}
\end{equation*}

\begin{prop}\label{prop:Stationary_phase}
Let $\sigma$ be some $(k,\rho)$-symbol, of order $\beta$. Then $T_k \sigma$ is in $S^{|\beta|}_\rho(Z)$ and we have the following expansion :
\begin{equation*}
T_k \sigma(x,\xi, h) \sim \sum_{\alpha \in \N^k} \frac{( i  h)^{|\alpha|}}{\alpha !} \left[\prod_1^k (\nabla_{x_i}.\nabla_{\xi_i})^{\alpha_i} \right] \sigma(x, \xi; x_1,\xi_1, \dots, x_k, \xi_k )_{|(x_i,\xi_i)=(x,\xi)}.
\end{equation*}
In addition, $T_k$ is continuous from the space of $(k,\rho)$-symbols to $S_\rho$.
\end{prop}

\begin{remark}\label{rem:terms_expansion}
All the terms in the expansion are in the right symbol class : if $\sigma$ is a $(k,\rho)$-symbol of order $\beta$, the terms with coefficient $h^{|\alpha|}$ are in finite number, and are symbols in $S^{|\beta| -|\alpha|}_\rho$.

We will only use this proposition for $k=1$ and $k=2$
\end{remark}

\begin{proof}
We prove the result by induction on $k$. First, if $k=0$, this is obvious, since $T_0 \sigma = \sigma$. Now, if we assume it is true for $k$, let $\sigma$ be some $(k+1,\rho)$ symbol. Then, we can consider that $\sigma$ is a $(k,\rho)$ symbol in its $k$ first coordinates, with the last coordinates as parameters. Applying $T_k$ on those first coordinates, we obtain that $T_k \sigma$ is a $(1,\rho)$ symbol by the induction hypothesis (here, the continuity of $T_k \sigma$ is important). Then, we remark that $T_{k+1} \sigma = T_1 T_k \sigma$.

Hence, proving the announced property for $T_1$ is sufficient. Assume for now that :
\begin{lemma}\label{lemma:basic_control}
If $\sigma$ is a $(1,\rho)$ symbol of order $(k_0,k_1)$ with $k_1 <0$, for some constant
\begin{equation*}
| T_1 \sigma(x,\xi,h)| \leq C \langle y\xi \rangle^{k_0+k_1}.
\end{equation*}
\end{lemma}

Let $\sigma$ be some $(1,\rho)$ symbol of order $(k_0, k_1)$. Changing variables $(v,V)=(x_1,\xi_1)-(x,\xi)$ in $T_1 \sigma$,
\begin{equation}\label{eq:changing_variables_T_1}
T_1 \sigma : (x, \xi, h) \mapsto (2\pi h)^{-(d+1)}\int  e^{\frac{ i }{h} \langle v,V \rangle} \sigma(x, \xi ;(x,\xi) + (v,V))  d v d V.
\end{equation}
Hence, the following identities hold :
\begin{align*}
X_y T_1 \sigma &= T_1(y_0\partial_{y_0} \sigma) + T_1(y_0\partial_{y_1}\sigma) \\
X_\theta T_1 \sigma &= T_1(y_0\partial_{\theta_0}\sigma) + T_1(y_0\partial_{\theta_1}\sigma) \text{ where $\theta_i$ is the $\theta$ coordinate for $x_i$} \\
X_Y T_1\sigma &= T_1\left(\frac{1}{y_0}(\partial_{Y_0} \sigma +\partial_{Y_1} \sigma)\right)\\
X_J T_1\sigma &= T_1\left(\frac{1}{y_0}(\partial_{J_0} \sigma +\partial_{J_1} \sigma)\right).
\end{align*}

We deduce then from \ref{lemma:basic_control} that $T_1\sigma$ is in $S_\rho$ with the correct order whenever $k_1 < 0$. Now, in the general case, we apply Taylor's formula to the $V$ variable :
\begin{equation*}
\begin{split}
\sigma(x,\xi; x+v, \xi + V) &= \\
\sum_{s = 0}^n \frac{1}{s !}  d _{\xi_1}^s \sigma( x,\xi;(x&+ v,\xi)).  V^{\otimes s}+ \int_0^1 \frac{(1-t)^n}{n !} ( d _{\xi_1}^{n+1} \sigma)(x,\xi ; x+v, \xi + tV) .V^{\otimes n+1}  d t.
\end{split}
\end{equation*}
Plugging this in the formula for $T_1\sigma$, integrating by parts in the $v$ variable, we obtain a sum
\begin{equation}\label{eq:asymptotic_expansion_first_variable}
\sum_{s=0}^n \frac{( i  h)^s}{s !} (\nabla_{x_1}.\nabla_{\xi_1})^s \left( \sigma(x,\xi ; x_1, \xi_1)\right)_{| (x_1, \xi_1) = (x, \xi)}
\end{equation}
and a remainder term
\begin{equation*}
(2\pi h)^{-(d+1)}( i  h)^{n+1}\int_0^1 \frac{(1-t)^n}{n !} \int e^{\frac{ i }{h} \langle v, V\rangle} \left[(\nabla_{x_1}.\nabla_{\xi_1})^{n+1} \sigma \right](x, \xi ; x+v, \xi + tV) d v d V   d t
\end{equation*}
Actually, after rescaling $V$ by a factor $t$, this remainder term is seen to be 
\begin{equation*}
( i  h)^{n+1}\int_0^1 \frac{(1-t)^n}{n !} T_1 \sigma^\ast (x,\xi, th)  d t
\end{equation*}
where $\sigma^\ast$ is of order $(k_0, k_1 - n- 1)$ and depends continuously on $\sigma$ (we only took a finite number of derivatives). If we take $n \geq k_1$, we already know that $T_1 \sigma^\ast$ is a symbol, so that the remainder is $\mathcal{O}(h^{n+1})$ is $S_\rho^n$, with constants depending on a finite number of derivatives of $\sigma$. Together with remark \ref{rem:terms_expansion}, this is enough to conclude.
\end{proof}

Now, let us prove lemma \ref{lemma:basic_control}.

\begin{proof} 
We rescale variable $V$ in \eqref{eq:changing_variables_T_1} to $W=V/h$, absorbing the $h^{-d-1}$ constant. Let $\chi\in C^\infty_c(\R^{d+1})$ equal $1$ near $0$, and break the integral into two parts with $1= (\chi + (1-\chi))(h^\rho yW)$. In the part with $1-\chi$, we can also introduce $1= (h^\rho yW)^{2N}/(h^\rho yW)^{2N}$ for some $N$ big enough, and get
\begin{equation*}
\int e^{ i  \langle v, W\rangle} \left[\chi(h^\rho y W) +(1-\chi(h^\rho y W))\frac{(h^\rho yW)^{2N}}{(h^\rho yW)^{2N}} \right] \sigma(x,\xi; x+v, \xi + hW) d v d W 
\end{equation*}
If we integrate the second term by parts in the $v$ variable $2N$ times, we get rid of the $(h^\rho yW)^{2N}$ on top. We see that for both terms we obtain an expression of the form
\begin{equation*}
\int e^{ i  \langle v, W\rangle} \psi(h^\rho yW) \sigma^\ast(x,\xi; x+v, \xi + hW) d v d W 
\end{equation*}
where either $(\sigma^\ast,\psi) = (\sigma, \chi)$ or $(\sigma^\ast, \psi) = (h^{2\rho N}y^{2N}(\partial_{y_1}^2 + \partial_{\theta_1}^2)^N\sigma,(1-\chi(x))/x^{2N})$. In both cases, $\sigma^\ast$ has the same properties as $\sigma$ (including support, bounds, and order), and $\psi$ is some symbol on $\R^{d+1}$ in the usual sense, of order $-2N$. We apply the same trick in the $v$ variable now, introducing
\begin{equation*}
1=\chi(h^{-\rho}v/y) + (1-\chi)(h^{-\rho}v/y)\frac{(h^{-\rho}v/y)^{2N}}{ (h^{-\rho}v/y)^{2N}}
\end{equation*} 
and integrating by parts in the $W$ variable for the second term. When differentiating $\psi$, the powers of $h$ compensate; when differentiating $\sigma^\ast$, we gain a positive power $h^{1-2\rho}$.In the end, we get new expressions of the form
\begin{equation*}
\int e^{ i  \langle v, W\rangle} \psi(h^\rho yW) \tilde{\psi}(h^{-\rho}v/y)\sigma^\ast(x,\xi; x+v, \xi + hW) d v d W 
\end{equation*}
where $\sigma^\ast$ still has the same properties as $\sigma$, and $\psi, \tilde{\psi}$ are some symbols on $\R^{d+1}$ in the usual sense, of order $-2N$. We can take the $L^1$ norm of the integrand, and it is bounded by :
\begin{align*}
 & C \int_{\R^{2(d+1)}} \langle h^\rho y W \rangle^{-2N} \langle h^{-\rho}v/y\rangle^{-2N} \langle y(\xi + h W) \rangle^{k_1} \langle y\xi \rangle^{k_0} dv dW. \\
\intertext{rescaling both $v$ and $W$, this is bounded by}
			& C \langle y\xi \rangle^{k_0} \int_{\R^{2(d+1)}} \langle v\rangle^{-2N} \langle  W \rangle^{-2N} \langle y\xi + h^{1-\rho} W \rangle^{k_1} dv dW \\
			&\leq C \langle y\xi \rangle^{k_0} \int_{\R^{d+1}} \langle  W \rangle^{-2N} \langle y\xi + h^{1-\rho} W  \rangle^{k_1} dW \quad \text{ for $N>d$.}
\end{align*}
We break the integral into two parts : $\{|W| > \varepsilon |y\xi|\}$ and $\{ |W| \leq \varepsilon |y \xi| \}$. Since $k_1<0$, $\langle y\xi + hW\rangle^{k_1} \leq 1$, and the first part is bounded by 
\begin{equation*}
C\langle y\xi\rangle^{k_0}\int_{|W| >\varepsilon |y\xi|}\langle W \rangle^{-2N} = \mathcal{O}(\langle y \xi \rangle^{k_0 + d- N + 1})=\mathcal{O}(\langle y \xi \rangle^{k_0 + k_1}) \text{ when $N \geq k_1 + d+1$.}
\end{equation*}
The second part is bounded by 
\begin{equation*}
C\langle y\xi\rangle^{k_0+k_1}\times \int_\R \langle W \rangle^{-2N}=\mathcal{O}(\langle y \xi \rangle^{k_0 + k_1}).
\end{equation*}

\end{proof}

\subsection{Symbolic calculus consequences}

We start this section by proving that the class of operator $\Psi_\rho$ is stable by composition.
\begin{prop}\label{prop:composition_lemma}
Let $a\in S^m_\rho(Z)$ and $b\in S^n_\rho(Z)$. Then, there is a symbol $c\in S^{m + n}_\rho(Z)$ and a negligible family of operators $R_h \in \mathcal{O}(h^{\infty})\Psi^{-\infty}$ such that 
\begin{equation*}
\Op(a)\Op(b) = \Op(c) + R_h
\end{equation*}
where
\begin{equation}\label{eq:asymptotic_expansion_product}
c(x,\xi) \sim \sum_{\alpha\in \N^2} \frac{(-1)^{\alpha_2}( i  h)^{|\alpha|}}{2^{|\alpha|}\alpha !} (\nabla_{x_1} . \nabla_{\xi_2})^{\alpha_1}(\nabla_{x_2} . \nabla_{\xi_1})^{\alpha_2} a(x_1, \xi_1)b(x_2, \xi_2) _{| x_1 = x_2 =x, \xi_1= \xi_2 = \xi}
\end{equation}
\end{prop}

\begin{proof}
First, we choose a truncation $\eta \in C^\infty_c(\R)$ equal to $1$ around the origin. Then $1-\eta$ is a symbol in $S(\R,0)$ vanishing around $0$, so we can apply proposition \ref{prop:support_around_y-diagonal}, and replace $\Op(a)$ and $\Op(b)$ by respectively $\Op(a)_{\eta}$ and $\Op(b)_{\eta}$. Recalling corollary \ref{cor:composition_negligible}, there exists $R_h \in \mathcal{O}(h^{\infty})\Psi^{-\infty}$ such that
\begin{equation*}
\Op(a)\Op(b) = \Op(a)_{\eta}\Op(b)_{\eta} + R_h.
\end{equation*}

If $K^w_\sigma$ is the kernel of $\mathbf{Op}^w_h(\sigma)$ on $\Hh^{d+1}$ as in \eqref{eq:Weyl-quantization-Rd}, we have
\begin{align*}
\sigma(x,\xi) &= \int e^{\frac{ i }{h}\langle u, \xi\rangle}  K^w_\sigma\left(x - \frac{u}{2}, x + \frac{u}{2}\right) d u . \\
\intertext{Since both $\mathbf{Op}^w_h(a)_{\eta}$ and $\mathbf{Op}^w_h(b)_{\eta}$ act on $\Hh^{d+1}$, the product also, and its kernel on $\Hh^{d+1}$ is}
K^w(x,\tilde{x}) &= \int K^w_a(x,x') K^w_b(x',\tilde{x}) \eta\left(\frac{y_{x'}}{y_x}-\frac{y_x}{y_{x'}}\right)\eta\left(\frac{y_{\tilde{x}}}{y_{x'}}-\frac{y_{x'}}{y_{\tilde{x}}}\right) dx'\\
\intertext{Hence, the solution to our problem is (formally) the function $c$ defined by}
c(x,\xi) &= h^{-2d-2}\int e^{\frac{ i }{h}\phi} \sigma_1(u,x', \xi_1, \xi_2)\chi\left(\frac{y+y_u/2}{y'},\frac{y-y_u/2}{y'}\right)du d\xi_1 dx' d\xi_2\\
\intertext{integrating over $(T^\ast \R^{d+1})^2$, where}
\phi &= \langle u,\xi\rangle + \langle x-u/2-x', \xi_1\rangle + \langle x'-x-u/2, \xi_2 \rangle \\
\sigma_1(u,x', \xi_1, \xi_2)	&= a\left(\frac{x-u/2+x'}{2}, \xi_1 \right) b\left(\frac{x+u/2+x'}{2}, \xi_2 \right)\\
\intertext{and $\chi$ is a smooth function on $\R^2$ supported in a rectangle } 
 &\{ (\tau,\kappa) \in \R^2 | 0< \epsilon \leq \tau \leq 1/\epsilon \quad 0< \epsilon \leq \kappa \leq 1/\epsilon\} \\
\intertext{After a change of variables, we will be able to use our stationary phase lemma. Let }
x_1 &= \frac{1}{2}(x + u/2 + x') \quad x_2 = \frac{1}{2}(x - u/2 + x') \\
\intertext{we get to write $c$ in the suitable fashion}
c(x,\xi) &=\left(\frac{2}{h}\right)^{-2d-2} \int e^{\frac{2 i }{h}( \langle x-x_1, \xi-\xi_1\rangle + \langle x-x_2, \xi-\xi_2\rangle)} \sigma_2 \chi_2 = T_2(\sigma_2 \chi_2)(x,\xi,h/2). \\
\intertext{where}
\sigma_2 &= a\left( x_2, \xi_1\right) b\left( x_1 , \xi_2 \right) \\
\chi_2	&= \chi\left(\frac{y_1-y_2 + y}{y_1 + y_2 - y},\frac{y_2 - y_1 + y}{y_1 + y_2 - y}\right).
\end{align*}

An elementary computation shows that $\chi_2$ is supported in some $\{\epsilon' y \leq y_{1,2} \leq y /\epsilon'\}$; hence, it is a smooth function of $y_{1,2}/y$ and it will have a good behavior w.r.t vector fields $y\partial_{y_{012}}$. Function $\sigma_2 \chi_2$ is supported in $\Omega_{2,\epsilon'}$. In addition, since the weights $(\langle y_i \xi_j \rangle)_{i=1,2}$ are equivalent to $\langle y_0 \xi_j \rangle$ in $\Omega_{2,\epsilon}$, $\sigma_2$ satisfies the desired estimates in that region, and $\sigma_2\chi_2$ is a $(2,\rho)$-symbol. From proposition \ref{prop:Stationary_phase}, we conclude directly that $c$ is in $S^{\alpha + \beta}_\rho(Z)$.
\end{proof}

\begin{prop}
The adjoint of $\Op_h(a)$ for the $L^2$ inner product is $\Op_h(\overline{a})$, so that real symbols yield self-adjoint operators, which is a key feature of the Weyl quantization.
\end{prop}

\begin{proof}
Taking $\mathcal{L} : f\in L^2 (Z) \to y^{-(d+1)/2} f \in L^2(dyd\theta)$, we see that 
\begin{equation*}
\Op(a) = \mathcal{L}^\ast \mathbf{Op}^w_h(a) \mathcal{L}.
\end{equation*}
Since the usual Weyl quantization on $\R^n$ has the property announced for $\Op$, we deduce the first part of the proposition: $\Op_h(a)^\ast = \Op_h(\overline{a})$.

Now, we use proposition 8.5 in appendix A in Taylor \cite{MR2744150}. It suffices to prove that when $a$ is real, $\Op_h(a) \pm i$ is surjective. Since $a$ is real, $a\pm i$ never vanishes, and we can find a symbol $\sigma^\pm_N$ such that 
\begin{equation*}
\Op(a\pm i) \Op(\sigma^\pm_N) = 1 + \mathcal{O}(h^N \Psi^{-N}).
\end{equation*}
for $h$ small enough, the operator on the LHS is invertible. In particular it is surjective, and so is $\Op_h(a\pm i)$.
\end{proof}

\begin{prop}
Let $a$ and $b$ be in $S_\rho(Z)$. Then, with $R_h \in \mathcal{O}(h^{\infty})\Psi^{-\infty}$,
\begin{equation*}
[\Op(a), \Op(b)] = \Op(c) + R_h
\end{equation*}
where $c$ is a semi-classical symbol with an asymptotic expansion with only odd powers of $h$, such that :
\begin{equation*}
c(x,\xi) = \frac{h}{ i } \{a, b\} + \mathcal{O}(h^{3(1-2\rho)} S^{n+ m - 3}_\rho)
\end{equation*}
where $\{.,.\}$ is the Poisson bracket defined with the symplectic form $ d \xi \wedge  d x$ :
\begin{equation*}
\{ f, g \} = \nabla_\xi f . \nabla_x g - \nabla_x f . \nabla_\xi g.
\end{equation*}
\end{prop}

\begin{proof}
Since in the asymptotic expansion \eqref{eq:asymptotic_expansion_product} the terms in odd powers of $h$ are symmetric in $a$ and $b$, this other key feature of Weyl quantization is now trivial.
\end{proof}

\begin{prop}\label{prop:optimal_L^2_regularity}
Let $a\in S^0_\rho(Z)$. Then $\Op(a)$ is bounded on $L^2$, with norm $\|a\|_{\infty} + o_{h\to 0}(1)$.
\end{prop}

\begin{proof}
We have all the ingredients to make the classical proof work. Consider
\begin{align*}
\| \Op_h(a) \|^2_{L^2} &= \| \Op_h(a)^\ast \Op_h(a) \|_{L^2}\\
\Op_h(a)^\ast \Op_h(a)	 &= \Op_h( |a|^2) + \mathcal{O}(hS^{n-1}_\rho) \text{ for }a\in S^n_\rho(Z)
\end{align*}
When $a$ has negative order, $|a|^2$ has a more negative order. Since $\Psi^{-d-2}$ operators are bounded with norm $O(h^{-\rho d -1})$, one can prove by induction that for any $\epsilon$, $\epsilon'$, $\Psi^{-\epsilon}$ operators are bounded on $L^2$ whenever $\epsilon>0$, with norm $\mathcal{O}(h^{-\epsilon'})$. Now, take $a\in S^0_\rho(Z)$. Let $M=\|a\|^2_\infty$. We just have to prove that $M - \Op_h(a)^\ast \Op_h(a) + o(1)$ is a positive operator. Take $\kappa>0$. Consider
\begin{equation*}
M + \kappa - \Op_h(a)^\ast \Op_h(a) = \Op_h( M + \kappa - |a|^2) + O(h^{1-\rho} S^{-1}_\rho).
\end{equation*}
But $M+ \kappa + |a|^2>\kappa$ so $b=\sqrt{M +\kappa + |a|^2}$ is in $S^0_\rho(Z)$ and real, so that $\Op(b)$ is self-adjoint, and
\begin{equation*}
M + \kappa + \Op_h(a)^\ast \Op_h(a) = \Op_h( b)^2 + \mathcal{O}(h^{1-\rho} S^{-1}_\rho) \geq -C.h^{1-\rho -\epsilon'} \text{ for any $\epsilon'$ given}.
\end{equation*}
We deduce that $M - \Op_h(a)^\ast \Op_h(a) \geq - C o(1)$.
\end{proof}

\begin{prop}\label{prop:functions_of_laplacian}
Let $f\in S(\R, n)$. With $P= -h^2 \Delta /2$, we define $f(P)$ by the spectral theorem. Recall that $p$ is the symbol of $P$. Then there is a symbol $\sigma$ such that
\begin{equation*}
\sigma = f\circ p + \mathcal{O}(h S^{n-1})
\end{equation*}
and 
\begin{equation*}
f(P) = \Op(\sigma) + R
\end{equation*}
where $R$ is \emph{asymptotically} negligible and commutes with $\partial_\theta$.
\end{prop}

\begin{proof}
If $f$ has positive order $n$, consider 
\begin{equation*}
f(P) = (P+i)^{n+1} \frac{ f}{(x + i)^{n+1}}(P).
\end{equation*}
Since $(P+i)^{n+1}$ is a pseudor, we only need to consider cases when $f$ has negative order. Following the method in lemma \ref{lemma:parametrix_1}, we get symbols $q_N(z)$ and $r_N(z)$ such that
\begin{equation*}
(P+z)\Op(q_N(z)) = \mathbb{1} + \Op(h^N r_N(z)).
\end{equation*}
What is more, the symbol norms of $q_N$ and $r_N$ are bounded by a power $|\Im z|^{-L_N}$ with $L_N \to \infty$ when $N \to \infty$. This establishes that $(P+z)^{-1}$ is a pseudor up to an asymptotically negligible remainder, for fixed $z$. Now, using a quasi-analytic extension of $f$ as in p.358 in \cite{MR2952218} (theorem 14.8 therein), and the bounds on $q_N$ and $r_N$, we see that the same can be said about $f(P)$.

To conclude, observe that $\Op(\sigma)$ commutes with $\partial_\theta$ whenever $\sigma$ does not depend on $\theta$, which is the case for $q_N$ and $r_N$.
\end{proof}

To prove a trace formula, it is convenient to be able to change quantizations.
\begin{lemma}\label{lemma:changing_quantization}
On $\R^{d+1}$, we can define a family of quantization as usual by
\begin{equation*}
\mathbf{Op}^t_h(\sigma)f(x) = \frac{1}{(2\pi h)^{d+1}} \int e^{i \langle x-x',\xi\rangle /h} \sigma\left(tx + (1-t)x', \xi \right)f(x')d\xi dx'
\end{equation*}
and then define $\Op^t_h(\sigma) := \mathcal{L}^\ast \mathbf{Op}^t_h(\sigma) \mathcal{L}$, for $\sigma \in S_\rho(Z)$ --- so that $\Op^{1/2} = \Op$.

If $a \in S^n_\rho$, there is a family $a_t$ of symbols so that for all $t\in [0,1]$,
\begin{equation*}
\Op^t(a_t) = \Op(a)  + \mathcal{O}(h^{\infty}\Psi^{-\infty}).
\end{equation*}
What is more
\begin{equation*}
a_t = a + \mathcal{O}(h^{1-\rho} S^{n-1}_\rho).
\end{equation*}
\end{lemma}

\begin{proof}
Following the scheme of proof of the composition lemma, we truncate the kernel of $\Op^{1/2}$ around $y=y'$ with a $\eta$ compactly supported, and we want to solve
\begin{equation*}
\Op^t(a_t)_\eta = \Op(a)_\eta.
\end{equation*}
If $K^t$ is the kernel of $\mathbf{Op}^t_h(a_t)$, we have
\begin{align*}
a_t(x,\xi) 	&= \int e^{ i  \langle u, \xi \rangle/h} K^t(x+ (t-1)u,x+tu) d u. \\
\intertext{so it is legitimate to consider }
b_t(x,\xi)	&:= (2\pi h)^{-d-1} \int e^{i \langle u , \xi- \xi'\rangle/h} a\left(x + (t-1/2) u, \xi' \right) \chi \left(\frac{y+ty_u}{y+ (t-1)y_u}\right) du d\xi'\\
			&= T_1 (\sigma_t)(x,\xi,h)
\end{align*}
with 
\begin{equation*}
\sigma_t(x, \xi; x_1, \xi_1) = a\left( (1/2+ t)x + (1/2 - t)x_1, \xi_1 \right) \chi\left(\frac{y(1+t) -ty_1}{ty + y_1(1-t)} \right)
\end{equation*}
and
\begin{equation*}
\chi(x) = \eta( x - 1/x).
\end{equation*}
One can check that $\sigma_t$ is a $(1,\rho)$-symbol. We deduce then from proposition \ref{prop:Stationary_phase} that
\begin{equation*}
\Op(a) = \Op^t(b_t) + \mathcal{O}(h^\infty\Psi^{-\infty}).
\end{equation*}
and
\begin{equation*}
b_t = a + \mathcal{O}( h^{1-\rho} S_\rho^{n-1}) \in S^n_\rho.
\end{equation*}
\end{proof}

Before we turn to a trace formula, observe that when one imposes Dirichlet conditions at $y=y_0$ and considers the Laplacian on $L^2( Z, \{y>y_0\})$, one finds that it has continuous spectrum $[d^2/4, +\infty)$. We cannot expect our operators to be trace class, if they are not even compact. This is why we introduce the following. 

Let $\Pi^\ast$ is the orthogonal projection (in $L^2(Z)$) on the non-zero Fourier modes in the $\theta$ direction. Also let $\Lambda'$ be the dual lattice to $\Lambda$ and $\Lambda'^\ast= \Lambda' \setminus \{ 0\}$.
\begin{lemma}\label{lemma:HS_cusp}
Let $\epsilon>0$ and $\kappa>0$. Let $\chi \in \mathscr{C}_b^\infty(Z)$ be supported in $\{ y > \kappa \}$. When $a\in S_\rho^{-(d+1)/2 - \epsilon}$ is supported in $\{y> \kappa\}$, both $\Op^1(a) \Pi^\ast$ and $\Pi^\ast \Op^0(a)$ are Hilbert-Schmidt. As a consequence, for any $A\in \Psi_\rho^{-(d+1)/2 -\epsilon}$, both $\Pi^\ast \chi A \chi$ and $\chi A \chi \Pi^\ast $ are Hilbert-Schmidt --- this is also true if $A$ is only asymptotically negligible.
\end{lemma}

\begin{proof}
The Hilbert-Schmidt (HS) norm (on $L^2(Z)$) of an operator $A$ with kernel $K$ w.r.t to the Lebesgue measure on the cylinder is
\begin{equation*}
\| A \|_{HS}^2 = \int_{Z\times Z} | K(x,x')|^2 \left(\frac{y'}{y}\right)^{d+1}  d y  d \theta  d y'  d \theta'.
\end{equation*}
Recall the Poisson formula (the covolume of $\Lambda$ is $1$)
\begin{equation*}
\sum_{\varpi\in \Lambda} e^{ i  \langle \varpi, W\rangle} = \sum_{W_i \in \Lambda'} \delta_{W_i}(W).
\end{equation*}
Using \eqref{eq:summing_kernel_quotient}, we deduce that the kernel of $\Op^1(a)$ is
\begin{equation*}
K^1_a(x,x') = (2 \pi)^{-d-1} \left(\frac{y}{y'}\right)^{\frac{d+1}{2}}\sum_{J \in \Lambda'}\int  e^{ i  \langle x-x',\xi \rangle} a\left( x, h \xi\right)d Y , \text{ where } \xi = (Y, J).
\end{equation*}
Seing this as a Fourier transform in the $x'$ variabla, by Parseval :
\begin{equation*}
\|\Op_1(a_1)\Pi^\ast \|^2_{HS} = \frac{1}{(2\pi)^d}\sum_{J\in \Lambda'^\ast} \int_{Z \times \R}   |a_1(y,\theta, hY, hJ)|^2  d y  d \theta  d Y
\end{equation*}
Since $a\in S^{-(d+1)/2 - \epsilon}_\rho$ is supported in $\{ y \geq \kappa\}$, this is less than
\begin{equation*}
C \sum_{J \in \Lambda'^\ast} \int_0^\infty \int_\R \langle h^2 y^2 (Y^2 + J^2)\rangle^{-d-1-2\epsilon} dY dy \leq C_N\sum_{J\in\Lambda'^\ast} \int_{\kappa}^\infty d y \frac{1}{hy (1+(hy|J|)^2)^{d/2+\epsilon}}.
\end{equation*}
After some change of variables, this is seen to be finite for fixed $h$. Observe that if we did not remove $J=0$, it would not be the case; there would be a logarithmic divergence. Hence $\Op^1(a) \Pi^\ast$ is HS. Taking adjoints, we see that that $\Pi^\ast \Op^0(\overline{a})$ is also HS.

Then, we write for some $N>0$ big enough. 
\begin{equation*}
(P+1)^{-(d+1)/4 - \epsilon}  = (1+\Op^1(h^N r^1_N))^{-1} \Op^1(q^1_N) = \Op^0(q^0_N)(1+\Op^0(h^N r^0_N))^{-1} 
\end{equation*}
We deduce that both
\begin{equation*}
\Pi^\ast \chi (P+1)^{-(d+1)/4 - \epsilon} \quad  \Pi^\ast (P+1)^{-(d+1)/4 - \epsilon}  \chi
\end{equation*}
are HS.

Take $R$ asymptotically negligible, supported in $\{ y> \kappa\}$. Then for any $N$, $(P+1)^N R$ is bounded on $L^2$, so that we can write $\Pi^\ast R = \Pi^\ast \tilde{\chi} (P+1)^{-N} \tilde{R}$ where $\tilde{\chi}$ is supported in some $\{y> \kappa'>0\}$, equal to $1$ on the support of $R$, and $\tilde{R}$ is a bounded operator for $h$ small enough. We deduce that $\Pi^\ast R$ is HS for $h$ small enough, and similarly for $R \Pi^\ast$.

Now, if $A\in \Psi^{-(d+1)/2 - \epsilon}_\rho$, it writes as $A= \Op^1(a) + \mathcal{O}(h^{\infty}\Psi^{-\infty})$, and
\begin{equation*}
\chi A \chi \Pi^\ast = \Op^1(\chi a)\chi \Pi^\ast + \chi \mathcal{O}(h^{\infty}\Psi^{-\infty})\chi \Pi^\ast
\end{equation*}
so that $\chi A \chi \Pi^\ast$ is HS.
\end{proof}

\begin{prop}\label{prop:Trace_cusp}
Take any $\epsilon>0$. Let $A\in \Psi^{-d-1-\epsilon}_\rho$ with principal symbol $a$. Then $\chi A\chi\Pi^\ast$ is trace class, and 
\begin{equation*}
Tr\left[ \chi A \chi \Pi^\ast \right] = \frac{1}{h^{d+1}}\left[\int_{T^\ast Z} \chi^2(x) a(x,h\xi) + \mathcal{O}(h + h^d|\log h|) \right].
\end{equation*}
If $A = \Op^1(a)$, then the remainder is improved to $\mathcal{O}(h^2 + h^d |\log h|)$.
\end{prop}

Observe that in the case of surfaces, the remainder is not as good as for compact manifolds.

\begin{proof}
First, if $R\in \mathcal{O}(h^\infty)\Psi^{-\infty}$ is supported in $\{y> \kappa\}$, then $R\Pi^\ast$ and $\Pi^\ast R$ are trace class since, for example, we can write 
\begin{equation*}
\Pi^\ast R = \Pi^\ast \chi(P+1)^{-N} \Pi^\ast (P+1)^{N} R
\end{equation*}
and this is the product of two HS operators.

Now, if $A\in \Psi^{-d-1-\epsilon}_\rho$, we can write 
\begin{equation*}
A = \Op^1(\tilde{a}) + \mathcal{O}(\Psi^{-\infty})
\end{equation*}
with $\tilde{a} = a + \mathcal{O}(h S^{-d-2-\epsilon}_\rho)$. Observe
\begin{equation*}
\chi \Op^1(\tilde{a})\chi\Pi^\ast = \left[\chi\Op^1(\tilde{a})\sqrt{\chi}(P+1)^{(d+1)/2 + \epsilon/2}\Pi^\ast\right]\left[ (P+1)^{-(d+1)/2- \epsilon/2}\Pi^\ast \sqrt{\chi}\right]
\end{equation*}
In the RHS, we have shown that the second term of the product is HS. Using proposition \ref{prop:functions_of_laplacian}, we can write the first term as $\Op^11(b) + R$ where $b \in S^{-(d+1)/2 -\epsilon/2}_\rho$ and $R$ is asymptotically negligible. It is thus also HS, and the product is trace class.

Writing the trace as the integral of the kernel along the diagonal, we obtain
\begin{equation*}
Tr \chi A \chi \Pi^\ast = \frac{1}{(2\pi)^{d+1}}\sum_{J\in \Lambda'^\ast}\int_{Z \times \R} \chi \tilde{a}(x,h\xi)dx dY.
\end{equation*}

To end the proof, we use:
\begin{lemma}\label{lemma:riemann_sums_for_symbols}
If $a\in S^{-d-1-\epsilon}_\rho$ is supported in some $\{ y>\kappa >0\}$, 
\begin{equation*}
\sum_{J\in \Lambda'^\ast}\int_{Z \times \R} a(x,h\xi)dx dY = \frac{1}{h^{d+1}} \left[\int_{T^\ast Z } a + \mathcal{O}(h^2 + h^d |\log h|)\right].
\end{equation*}
and both sums converge absolutely.
\end{lemma}
\end{proof}

The proof of lemma \ref{lemma:riemann_sums_for_symbols}:
\begin{proof}
Let $D'$ be fundamental domain for the action of $\Lambda'$ on $\R^d$. Assume $D'$ to be symmetric around $0$, and of bounded diameter. Its volume is $1$. Then, for $\varpi\in \Lambda'^\ast$, 
\begin{equation*}
\left|f(h\varpi)  - \frac{1}{h^d}\int_{h\varpi + h D'} f(J)d J\right| \leq C \frac{h^2}{h^d} \int_{h\varpi + hD'} \|  d ^2_J f(J) \|.
\end{equation*}
hence the difference between the two main terms in lemma \ref{lemma:riemann_sums_for_symbols} is bounded up to some constant by
\begin{equation*}
h^{-d-1}\int_{(x,\xi)\in T^\ast Z, J\in hD'} |a(x,\xi)| + h^{1-d}\int_{(x,\xi)\in T^\ast Z, J\notin hD'}  \| d ^2_J a\|.
\end{equation*}

For the first term, first integrate variable $\theta$ (losing a constant $\mathrm{vol}(D)$) and then variable $Y$ after rescaling. Obtain a bound by
\begin{equation*}
h^{-d-1} \int_{J \in hD', y>\kappa} \frac{1}{y} (1+y^2J^2)^{\frac{-d-\epsilon}{2}}
\end{equation*}
Rescaling the $y$ variable, this is bounded by (note the use of polar coordinates in $J$),
\begin{equation*}
h^{-d-1} \int_0^{Ch} r^{d-1}  d r \int_{\kappa r}^{+\infty} \frac{ d y}{y(1+y^2)^{\frac{d+\epsilon}{2}}}.
\end{equation*}
This is $\mathcal{O}(h^{-1}|\log h|)$. Likewise for the second term, it is bounded by :
\begin{equation*}
h^{1-d} \int_{Ch}^\infty r^{d-1}  d r \frac{1}{r^2}\int_{\kappa r}^{+\infty}\frac{y d y}{(1+y^2)^{\frac{d+\epsilon}{2}}}
\end{equation*}
This is $\mathcal{O}(h^{1-d})$ if $d>2$ and respectively when $d=1$, $\mathcal{O}(h^{-1})$, and when $d=2$, $\mathcal{O}(h^{-1} |\log h|)$.
\end{proof}

\section{Applications}

Now we will present some applications of the cusp-quantization.
\subsection{Cusp manifolds}

\subsubsection{Quantization}
As we said in the introduction, cusp manifolds are described as a compact manifold with boundary to which is glued a finite number of cusps. Here, we give a formal definition that will simplify the construction of the quantization :
\begin{definition}
Let $(M, g)$ be a complete $(d+1)$-dimensional riemannian manifold. $M$ is said to be a cusp manifold if it is endowed with a cusp atlas $\mathcal{F}$, that is 
\begin{itemize}
	\item a finite collection $(U_i, U'_i, \gamma_i)_i$ of $\R^{d+1}$-charts, that is, diffeomorphisms $\gamma_i : U_i\subset M \to U'_i\subset \R^n$, with $U_i$ relatively compact.
	\item a finite collection $ (Z_j, Z'_j, \gamma^c_j)_j $ of cusp-charts, that is, diffeomorphisms $\gamma^c_j : Z_j\subset M \to Z'_j \subset Z_{\Lambda_j}$ such that $\gamma^c_j$ is an isometry, and $Z'_j$ is of the form $\{ y> a_j\}$.
\end{itemize}
We require that
\begin{itemize}
	\item No two $X_j$'s intersect.
	\item The coordinate changes between two $(U'_i)$'s or $Z_j$ and $U_i$ are diffeomorphisms.
	\item The lattices $\Lambda_i$ have covolume $1$. This is a convention, and there is only one choice of height function $y_i$ that is coherent with that choice.
\end{itemize}
\end{definition}

\begin{definition}\label{def:symbols_manifold}
In $\R^{d+1}$, we define the \emph{Kohn-Nirenberg} symbols of order $n$, in the usual way, as in \cite[p.207]{MR2952218}: $\sigma \in S^n_\rho(\R^{d+1})$ whenever for all $k,k'\geq 0$ there is a constant $C_{k,k'}>0$, 
\begin{equation*}
|d^{k'}_{x} d^k_{\xi} \sigma| \leq C_{k,k'} \langle \xi \rangle^{n-k}, \text{ for } x,\xi \in \R^{d+1}
\end{equation*}

The class $S^n_\rho(M)$ of hyperbolic symbols of order $(n,\rho)$ is composed of the functions $\sigma$ on $T^\ast M$ such that for any chart $(U,V\subset N, \gamma)$ in the atlas, the function 
\begin{equation*}
\sigma^U(x,\xi) : = \gamma^\ast \sigma \left[=\sigma(\gamma^{-1}(x),  d \gamma(\gamma^{-1}(x))^\ast \xi) \right]
\end{equation*}
is the restriction to $T^\ast V$ of some element of $S^n_\rho (N)$ (with $N= \R^{d+1}$ or $N= Z_{\Lambda_i}$). The invariance by coordinate changes of the Kohn-Nirenberg class \cite[theorem 9.4, p.207]{MR2952218} implies that this is well defined --- it does not depend on the choice of the atlas.
\end{definition}

To define a quantization on cusp manifolds that enjoys all usual properties, we follow the procedure in p. 347 through to p. 352 in \cite{MR2952218}. A pseudo-differential operator on $M$ is defined as an operator $C^\infty_c(M) \to C^\infty(M)$ such that restricted to any chart, it is pseudo-differential --- in the case of a cusp-chart, this means that it is in some $\Psi_\rho(Z_{\Lambda_i})$. We also require that they are pseudo-local --- that is, when we truncate their kernel at a fixed distance of the diagonal, we obtain negligible operators. 

Lemma \ref{lemma:Pseudo-differential_behavior} proves that it suffices to check the above properties for the finite set of charts of some cusp-atlas. Lemmas \ref{lemma:Pseudo-differential_behavior} and \ref{prop:support_around_y-diagonal} ensure that the class of pseudors is not reduced to compactly supported operators, because pull backs of elements of $\Psi(Z_{\Lambda_i})$ \emph{are} pseudo-local. 

We can define the semi-classical principal symbol $\sigma^0(A)$ of a pseudor $A$ as for pseudors on compact manifolds --- once again thanks to lemma \ref{lemma:Pseudo-differential_behavior} --- and according to the definition, $\sigma^0(A)$ is in some $S^n_\rho(M)$. The class of $A$ such that $\sigma^0(A)\in S^n_\rho(M)$ is denoted $\Psi^n_\rho$. We let $\Psi_\rho^{-\infty}(M)$ be the class of smoothing operators in the Sobolev sense --- as in definition \ref{def:smoothing_cusp}. Let $\Psi_\rho(M) = \cup_{n\geq -\infty} \Psi^n_\rho$. When we omit the $\rho$, we refer to the case $\rho = 0$.

Using charts, our quantization $\Op$ in $Z_{\Lambda_i}$ and the usual Weyl quantization on $\R^{d+1}$, we are able to build a quantization procedure $\Op$ on $M$, that is a section to the symbol map. Using classical results, and the first part of the article, we see that $\Psi^0_\rho$ gives bounded operators on $L^2(M)$, whose norm is the $L^\infty$ norm of the symbol, up to a $o(1)$ term as $h\to 0$ --- that could be estimated with derivatives of the symbol.

For a height $a$ bigger than all the $a_j$, we define $\Pi^\ast_a$ as the projection on non zero Fourier modes in $\{y>a\}$ :
\begin{equation*}
\Pi^\ast_a f := f - \mathbb{1}(y>a) \int f d \theta .
\end{equation*}
The following hold :
\begin{prop}
Let $A\in \Psi^{-1}_\rho(M)$. Then $\Pi^\ast_a A$ is compact on $L^2$. 

Let $A\in \Psi^{-n}_\rho(M)$ with $2n > d+1$. Then $\Pi^\ast A \Pi^\ast$ is Hilbert-Schmidt. 

Let $A\in \Psi^{-n}_\rho(M)$ with $n > d+1$. Then $\Pi^\ast A \Pi^\ast$ is trace class, and
\begin{equation*}
Tr \Pi^\ast_a A \Pi^\ast_a = \frac{1}{h^{d+1}} \left[\int_{T^\ast M} \sigma(A) + \mathcal{O}(h + h^d\log h) \right]
\end{equation*}
\end{prop}

\begin{proof}
In the compact part of $M$, these are classical results --- see theorem 4.28 p. 89, and remark (C.3.6) p. 412 in \cite{MR2952218}; see also proposition 9.2 and theorem 9.5 p.112 and following in \cite{MR1735654}. So we only need to prove this when $A$ is only supported in the cusps, and for negligible operators. For the Hilbert-Schmidt and the trace-class property, this are the contents of lemma \ref{lemma:HS_cusp} and proposition \ref{prop:Trace_cusp} when $A$ is supported only in the cusps. 

As to negligible operators, they can always be written down as the product of another negligible operator with some power of $(P+1)^{-1}$, and the arguments we used in the proof of lemma \ref{lemma:HS_cusp} and proposition \ref{prop:Trace_cusp} will carry on, so that what we really need to prove is that $\Pi^\ast_a A$ is compact on $L^2(M)$ when $A\in \Psi^{-1}_\rho$.

But that is a consequence of the fact that $\Pi^\ast_a H^1(M)$ is compactly injected in $L^2(M)$. Once again, as this is always true for compact manifolds with boundary (Rellich's theorem), it suffices to prove it for for the cusps. More precisely, we need to show that $\{ \mathbb{1}_{y>a}f \; | \; f \in H^1(Z_\Lambda), \Pi^\ast f = f \}$ is compactly injected in $L^2(Z_\Lambda)$. We recall the proof from \cite{MR0562288} --- see pp. 206 and following. Consider the fact that $f\in H^1(Z_\Lambda) \mapsto \mathbb{1}_{a< y< T} \Pi^\ast f\in L^2(Z_\Lambda)$ is compact. Now, using the Wirtinger inequality in the torus, one can prove 
\begin{equation*}
\| \mathbb{1}_{y>T} f\|_{L^2(Z_\Lambda)} \leq \frac{C}{T^{d+1}} \|\mathbb{1}_{y>T} \nabla f \|_{L^2(Z_\Lambda)}.
\end{equation*}
This proves that the mapping $\mathbb{1}_{y>a}\Pi^\ast : H^1(Z_\Lambda) \to L^2(Z_\Lambda)$ is the norm limit of a sequence of compact operators, so it is compact.
\end{proof}

\subsubsection{Egorov lemma for Ehrenfest times}
In this section, we give an Ehrenfest time Egorov lemma, which was the original motivation for what we have done so far. The Levi-Civita connexion on $M$ is associated to a splitting of $T T^\ast M= V \oplus H$. $V$ and $H$ are subbundles that can be identified with respectively $T^\ast M$ and $TM$. The only metric on $T^\ast M$ that renders $V$ orthogonal to $H$ and that makes those identifications isometries is called the Sasaki metric. It is in some sense the natural metric to use on $T^\ast M$ for our problem; we recall a few facts on it in appendix \ref{appendix:Sasaki}. Now the we have specified a riemannian metric on $T^\ast M$, we can define the spaces $\mathscr{C}^k(T^\ast M)$ as in appendix \ref{appendix:Sasaki}. Let us introduce a particular class of symbols :
\begin{definition}
Let $U$ be some open set of $\R^2$. For $E>0$, let $S^E_{C}$ denote the class of functions $\sigma$ on $U \times T^\ast M$ that are $\mathscr{C}^\infty(T^\ast M)$ in the second variable, supported in $(T^\ast M)_E :=\{ p\leq E \}$. Additionally require there are constants $C_k > 0$ such that
\begin{equation*}
\| \sigma(h,\tau; .) \|_{\mathscr{C}^k(T^\ast M)} \leq C_k e^{C k |\tau|}
\end{equation*}
where $(h,\tau)$ are the coordinates in $\R^2$. 
\end{definition}
From proposition \ref{prop:equivalence_norms_Sasaki}, elements of $S^E_C$ are symbols in $S^{-\infty}$ for fixed $t$. Additionally, if the open set $U$ is $\{ C |\tau| \leq \rho |\log h| \}$ with $ \rho < 1/2$, elements of $S^E_C$ are symbols in $S^{-\infty}_\rho$, and can be quantized. We will assume that $U$ takes this form in the rest of the article.

Let us point out that when $A$ is in some $\Psi^n(M)$, and $\sigma \in S^E_C$, up to a negligible operator $R$, 
\begin{equation*}
[A, \Op(\sigma) ] - \frac{h}{i} \Op(\{ \sigma(A), \sigma \}) = \Op(\tilde{\sigma}) + R
\end{equation*}
where $\tilde{\sigma}$ is $\mathcal{O}((he^{C|\tau|})^2)$ in $S^E_C$.

Let us introduce
\begin{definition}\label{def:max_lyapunov}
The \emph{maximal Lyapunov exponent} of the geodesic flow on $(T^\ast M)_E$ is defined as
\begin{equation*}
\lambda_{max}(E) := \sup_{\xi \in (T^\ast M)_E}\limsup\limits_{t\to \infty} \frac{1}{|t|}\log \|  d_\xi \varphi_t \|.
\end{equation*}
\end{definition}
Using Jacobi fields and Rauch's comparison theorem --- see 1.28 in section 1.10 of \cite{MR0458335} --- one can prove that $\lambda_{max}(E)$ is bounded by $E \kappa$ where $-\kappa$ is the minimum of the curvature of $M$. Observe that proposition \ref{prop:estimating_derivatives_propagation_flow} implies that for any $\lambda > \lambda_{max}(E)$, and any $f\in \mathscr{C}^{\infty}(T^\ast M)$ supported in $(T^\ast M)_E$, $f\circ \varphi_t$ is in $S^E_{\lambda}$.

Recall that the Schr\"odinger propagator is 
\begin{equation*}
U(t) = e^{-it P/h}
\end{equation*}
We have
\begin{theorem}\label{thm:Egorov_lemma}
Let $\sigma \in \mathscr{C}^{\infty}(T^\ast M)$ be supported in $(T^\ast M)_E$. Then, for any $\rho < 1/2$ and any $\lambda > \lambda_{max}(E)$, there exists a symbol $\tilde{\sigma}_\rho$ that is in $S^E_{\lambda}$, with $U=\{ |\tau| \leq \rho |\log h|/\lambda\}$. On $U$,
\begin{equation*}
\tilde{\sigma}(t, x, \xi) = \sigma(\varphi_t(x,\xi)) + \mathcal{O}(h |t| e^{2\lambda |t|}),
\end{equation*}
and
\begin{equation*}
U(-t)\Op(\sigma)U(t) = \Op(\tilde{\sigma}) + \mathcal{O}( (|t| h e^{2\lambda |t|})^\infty)
\end{equation*}
where the remainder is asymptotically smoothing.
\end{theorem}

Since Beal's theorem --- see theorem 8.3 in \cite{MR2952218} --- is not available to us, we can only prove that the remainder is \emph{asymptotically} smoothing.

\begin{proof}
Let us assume that we found an exact solution $\tilde{\sigma}$. Then, we would have :
\begin{align*}
\Op(\tilde{\sigma}) &= e^{ i  t P/h} \Op(\sigma) e^{- i  t P /h} \\
\intertext{i.e}
\Op(\sigma) 		&= e^{- i  t P/h} \Op(\tilde{\sigma}) e^{ i  t P /h}. \\
\intertext{Differentiating with $t$,}
0					&= e^{- i  t P/h}\left[ \Op(\partial_t \tilde{\sigma}) - \frac{ i }{h} [P, \Op(\tilde{\sigma})] \right] e^{ i  t P/h}
\end{align*}

All along our development, we will follow the proof in \cite{MR2952218} closely. Let us build by induction a family of operators
\begin{equation*}
B_n(t) = \Op(b_n) \quad , \quad E_n(t) = \Op(e_n)
\end{equation*}
where $b_n$ and $e_n$ are in $S^E_\lambda$, satisfying :
\begin{align*}
\frac{h}{ i }\partial_t B_n &= [P, B_n] + E_n +  R_n.\\
B_n(0)							&= \Op(\sigma)
\end{align*}
the remainder $R_n$ being negligible, and with the estimates :
\begin{align*}
b_n - b_{n-1} &= \mathcal{O}_{S^E_\lambda}((|t|h)^n e^{2n\lambda |t|}) \text{ for $n>0$}\\
e_n			&= \mathcal{O}_{S^E_\lambda}(h^{2+n} |t|^n e^{(2n+2)\lambda |t|}).
\end{align*}

For $n=1$, define 
\begin{equation*}
B_0 = \Op(\sigma\circ \varphi_t).
\end{equation*}
This is $\mathcal{O}(1)$ in $S^E_\lambda$. Then
\begin{align*}
\frac{h}{ i }\partial_t B_0 &= \frac{h}{ i } \Op(\{p, \sigma\circ\varphi_t\}) \\
		&= [P, B_0] + E_0 + R_0,
\end{align*}
where $R_0$ is negligible and $E_0 = \Op(e_0)$. From the product formula, we get that $e_0$ is still supported in $(T^\ast M)_E$, and it is $\mathcal{O}( h^2 e^{2\lambda |t|})$ in $S^E_\lambda$.

Assume that all the assumptions hold for some $n \geq 0$, and let
\begin{align*}
c_{n+1} &= \frac{ i }{h}\int_0^t e_n(s) \circ \varphi_{t-s}  d s.\\
\intertext{$c_{n+1}$ is in $S^E_\lambda$, and it is $\mathcal{O}((|t| h)^{n+1} e^{(2n+2)\lambda |t|})$. One gets}
\frac{h}{ i }\partial_t \Op(c_{n+1}) &= \frac{h}{ i }\Op(\{p, c_{n+1}\} + \frac{ i }{h} e_n ) \\
	&= [P, \Op(c_{n+1})] + E_n - E_{n+1} + R_{n+1}\\
\intertext{where}
E_{n+1}&= \Op(e_{n+1}) \text{ with } e_{n+1}= O_{S^E_\lambda}( h^{3+n} |t|^{n+1} e^{(2n+4)\lambda|t|}) \\
\intertext{At last, define}
B_{n+1} &= B_n - C_{n+1}
\end{align*}
Such $b_n$'s and $e_n$'s satisfy the announced properties. Now, since
\begin{equation*} 
\frac{h}{ i } \partial_t [e^{- i  t P/h} B_n e^{ i  t P/h}] = e^{- i  t P/h} (E_n + R_n) e^{ i  t P/h}
\end{equation*}
Duhamel's formula gives :
\begin{equation*}
\begin{split}
B_n (t) &= e^{ i  t P/h} \Op(\sigma) e^{- i  t P /h} \\
		& {} - \frac{ i }{h} \int_0^t U(s-t)[E_n(s) + R_n(s) ] U(t-s)  d s.
\end{split}
\end{equation*}
Since the operators $U(t)$ are bounded from $H^s$ to $H^s$ for any $s$, we deduce that this is $\mathcal{O}((|t| h e^{2\lambda |t|})^{n+1}h^{-2N})$ in $H^{-N}\to H^N$ operator norm. We can find a symbol $\tilde{\sigma} \in S^E_\lambda$ such that
\begin{equation*}
\tilde{\sigma} \sim b_0 - \sum_1^\infty c_n = \sigma \circ \varphi_t + \mathcal{O}_{S^E_\lambda}( h|t|e^{2\lambda |t|}).
\end{equation*}

Then, $\tilde{\sigma}$ satisfies the condition of the theorem.

\end{proof}

\begin{remark}\label{remark:support_Egorov_lemma}
Following the support of the $b_n$'s, $e_n$'s, we find that $\tilde{\sigma}$ is exactly supported in $\varphi_t(\supp(\sigma))$. Actually, the whole operator is microsupported on that set ; if we multiply our conjugated operator by some $\Op(\eta)$ such that $\eta$ vanishes on $\supp{\tilde{\sigma}}$, we obtain a negligible operator (not only asymptotically).
\end{remark}

\subsection{Extending a result of Semyon Dyatlov}

\subsubsection{Spectral theory and Eisenstein functions}\label{section:spectr_theory}
The following facts on the spectral theory of the Laplacian on cusp-manifolds are contained in \cite{MR725778}. However, in that article, M\"uller considered cusps where the horizontal slices were arbitrary compact $d$-dimensional manifolds instead of tori, so that his definition of \emph{Riemannian manifolds with cusps} is more general than our cusp-manifolds. However, he also wrote an article in the case of surfaces \cite{} with the same definition of cusp, which is a good place to start if one wants to learn about cusp surfaces.

The non-negative Laplacian $-\Delta$ acting on $C_0^\infty(M)$ functions has a unique self-adjoint extension to $L^2(M)$ and its spectrum consists of
\begin{enumerate}
\item Absolutely continuous spectrum $\sigma_{ac}=[d^2/4, +\infty)$ with multiplicity $k$ (the number of cusps).
\item Discrete spectrum $\sigma_d =\{\lambda_0 = 0 < \lambda_1 \leq \dots \leq \lambda_i \leq \dots \}$, possibly finite, and which may contain eigenvalues embedded in the continuous spectrum. To $\lambda \in \sigma_d$, we associate a family of orthogonal eigenfunctions that generate its eigenspace $(u_\lambda^{i})_{i=1 \dots d_\lambda} \in L^2(M) \cap C^\infty(M)$.
\end{enumerate}

The generalized eigenfunctions associated to the absolutely continuous spectrum are the Eisenstein functions, $(E_j(x,s))_{i=1 \dots k}$. Each $E_j$ is a meromorphic family (in $s$) of smooth functions on $M$. Its poles are contained in the open half-plane $\{\Re s < d/2\}$ or in $(d/2, 1]$. The Eisenstein functions are characterized by two properties : 
\begin{enumerate}
\item $\Delta_g E_j(.,s) = s(d-s)E_j(.,s)$ 
\item In the cusp $Z_i$, $i=1 \dots k$, the zeroth Fourier coefficient of $E_j$ in the $\theta$ variable 
equals $\delta_{ij} y_i^s + \phi_{ij}(s) y_i^{1-s}$ where $y_i$ denotes the $y$ coordinate in the cusp $Z_i$ and $\phi_{ij}(s)$ is a meromorphic function of $s$.
\end{enumerate}

Let us recall the construction of the Eisenstein functions. On $M$ we define a function $y_M$ that corresponds to $y_i$ on $X_i \cap \{ y_i \geq 2a \}$, and equals $1$ on $M_0$. Let $\chi$ be a smooth monotonous function that equals $1$ on $[3a, +\infty [$, and vanishes on $]-\infty, 2a]$. We let $\chi_i$ be the function supported in cusp $Z_i$, where it is $\chi \circ y_i$. Now, let $\tilde{\chi}\in C^\infty_c(\R, [0, 1])$ such that $\tilde{\chi} \equiv 1$ on $]-\infty, \ln 4]$ and $\tilde{\chi} \equiv 0$ on $[\ln 5, +\infty[$. For $s\in \R^+$, let
\begin{equation}\label{def:sharp_cutoffs}
 \chi_s := \tilde{\chi}\left(\ln\left(\frac{y_M}{a}\right) - s \right).
\end{equation}

Take $E^0(s,x)=y_M^s$. Then consider 
\begin{equation*}
E_i(s,x):= \chi_i E^0(s,x) + (-\Delta - s(d-s))^{-1} [\Delta, \chi_i] E^0(s,x).
\end{equation*}
Since $\chi'$ is compactly supported, $[\Delta, \chi_i]E^0(s,.)$ is compactly supported and in $L^2$, so this is well defined. One can check that 
\begin{equation*}
(-\Delta - s(d-s)) E_i = -[\Delta, \chi_i] E^0 + [\Delta, \chi_i] E^0 = 0.
\end{equation*}
to see that the $E_i$'s satisfy the announced properties. Uniqueness is then straightforward. In what follows, we will use the notations :
\begin{align*}
s	&= d/2 +  i  /h + \eta(h)\\
W 	&= \frac{h^2}{2}s(d-s) = \frac{h^2}{2}\left[\frac{d^2}{4} + \frac{1}{h^2} - \eta^2 - 2 i  \frac{\eta}{h}\right] \\
	&= \frac{1}{2}\left[1 -2 i  \eta h + h^2\left(\frac{d^2}{4}-\eta^2\right)\right]. 
\end{align*}

Let us define the measures $\mu_{i,\eta}$ announced in the introduction. For $f\in C^0_c(T^\ast M)$ compactly supported, let
\begin{equation*}
\mu_{i,\eta}^\pm(f) := 2\eta a^{2\eta} \int_{\R\times \mathbb{T}_\Lambda} e^{-2\eta t} f\circ\varphi_{-\pm t}(a,\theta, \pm 1/a, 0) d t d \theta
\end{equation*}
This defines two Radon measures. We also recall the definition of the Wigner distributions
\begin{equation*}
\langle \mu_{i,j}^h(s),\sigma \rangle := \langle \Op(\sigma) E_i(s), E_j(s)\rangle \text{ for $\sigma\in C^\infty_c(T^\ast M)$.}
\end{equation*} 
We will prove the following theorem
\begin{theorem}\label{theorem:Extension_Semyon}
Consider $s_h = 1/2 \pm i/h + \eta(h)$. All the limits are taken when $h\to 0$.
\begin{enumerate}
	\item If $\eta(h) \to \nu > 0$, then $\eta\mu_{i,j}^h(s_h) \rightharpoonup \delta_{i,j}\pi \mu^\pm_{i,\nu}$ in $C^\infty_c(T^\ast M)'$.
	\item Assume that $M$ has negative curvature. Whenever $\eta \to 0$ with 
\begin{equation*}
\liminf \;\eta \frac{ |\log h|}{ \log |\log h|} > \lambda_{max}\left( \frac{1}{2} \right),
\end{equation*}
then $\eta \mu_{i,j}^h(s_h) \rightharpoonup \delta_{i,j}\pi\mathscr{L}_1$.
\end{enumerate}
\end{theorem}

The case when $\eta \to \nu >0$ was proven in dimension $2$ by Semyon Dyatlov in \cite{MR2913877}. We cautiously follow the steps of his proof, paying attention to the constants. The long time Egorov lemma is really what enables us to extend S. Dyatlov's result and get (2).

The proof is divided into three parts. We first approximate the Eisenstein series by a Lagrangian state propagated by the Schr\"odinger flow. Such an approximation cannot work near the spectrum, and that is why the approach taken here probably cannot be improved to capture resonances arbitrarily close to the spectrum. Then, we use the Egorov lemma to reduce the problem to a stationary phase computation in the cusp. The last part of the proof is a dynamical argument, from Babillot; we essentially prove that incoming horocycles from the cusp equidistribute in $M$.

It suffices to consider the case $\Im s \to + \infty$, the other can be deduced thereof.

\subsubsection{Reduction to a lagrangian expression}
We fix an exponent $\lambda > \lambda_{max}(1/2)$. Observe that since the hamiltonian $p$ of the geodesic flow is 2-homogeneous, $\varphi_t(\kappa \xi) = \kappa \varphi_{\kappa t}(\xi)$. Consider $\Phi_\kappa : T^\ast M \to T^\ast M$ the multiplication by $\kappa$. Then 
\begin{equation*}
d \varphi_t = d\Phi_\kappa \circ d \varphi_{\kappa t}\circ  d \Phi_\kappa^{-1}.
\end{equation*}
If $\kappa \geq 1$, we have $\|d\Phi_\kappa \| = \kappa$ and $\| d\Phi_\kappa^{-1} \| = 1$ (by inspecting the behaviour of $d\Phi_\kappa$ on the vertical and horizontal bundles of $T^\ast M$). We deduce that 
\begin{equation*}
\lambda_{max}(\kappa E) = \kappa \lambda_{max}(E).
\end{equation*}
It follows that for any $\epsilon > 0$ sufficiently small, $\lambda > \lambda_{max}(E = 1/2 + \epsilon)$.

Let us take $T>0$ such that $\sigma$ is supported in $\{y_M \leq ae^T \}$. We aim to replace $E_i(s)$ on the support of $\sigma$ by a propagated incoming wave. That is why we define :
\begin{align*}
\tilde{E}^0_i(s,t) & = \chi_{T-\ln 3} e^{\frac{ i }{h}t(P-W)} \chi_{T+t} \chi_i E^0(s) \\
\tilde{E}_i(s,t)   & = \chi_{T-\ln 3} e^{\frac{ i }{h}t(P-W)} \chi_{T+t} E_i (s) 
\end{align*}
and prove :
\begin{lemma}\label{lemma:approx_by_lagrangian} When $\eta$ remains bounded,
\begin{equation*}
\| \chi_{T-\ln 3} E_i(s) - \tilde{E}^0_i(s,t) \|_{L^2} = O\left(\frac{e^{-\eta t}}{\eta}\right) + \mathcal{O}( (|t| he^{2\lambda|t|})^\infty).
\end{equation*}
\end{lemma}

\begin{proof}
We write 
\begin{equation*}
\chi_{T-\ln 3} E_i - \tilde{E}^0_i =  (\chi_{T-\ln 3} E_i - \tilde{E}_i) + (\tilde{E}_i - \tilde{E}^0_i).
\end{equation*}
Then, we prove successively 
\begin{lemma}\label{lemma:approx_by_lagrangian_1}
\begin{equation*}
\tilde{E}_i - \tilde{E}^0_i= O_{L^2}\left(\frac{e^{-\eta t}}{\eta}\right).
\end{equation*}
\end{lemma}
and
\begin{lemma}\label{lemma:approx_by_lagrangian_2}
\begin{equation*}
\chi_{T-\ln 3} E_i - \tilde{E}_i   = O_{L^2}(( |t| he^{2\lambda |t|})^\infty).
\end{equation*}
\end{lemma}
\end{proof}

we start with lemma \ref{lemma:approx_by_lagrangian_1}.
\begin{proof}
We have 
\begin{equation*}
\tilde{E}_i - \tilde{E}^0_i = \chi_{T-\ln 3} e^{\frac{ i  t}{h} (P-W)} \chi_{T+t} (-\Delta - s(d-s))^{-1}[\Delta, \chi_i ] E^0
\end{equation*}
Thus 
\begin{equation*}
\|\tilde{E}_i - \tilde{E}^0_i\|_{L^2} \leq e^{-\Re \frac{ i  t W}{h}} \| (-\Delta - s(d-s))^{-1}\|_{L^2 \to L^2} \| [\Delta, \chi_i ] E^0\|_{L^2}.
\end{equation*}
since $\Delta$ is self adjoint, we have
\begin{equation*}
\| (-\Delta - s(d-s))^{-1}\|_{L^2 \to L^2} \leq \frac{h}{2\eta}.
\end{equation*}
What is more, $\Re ( i  t W) = h\eta t$. Now, 
\begin{equation*}
[\Delta, \chi_i] E^0 = (\Delta \chi_i)E^0 + 2y s \partial_y \chi_i E^0  = O_{L^2}\left(\frac{1}{h}\right).
\end{equation*}
putting all three inequalities together, we conclude.
\end{proof}
we go on to lemma \ref{lemma:approx_by_lagrangian_2}.
\begin{proof}
When $t=0$, 
\begin{equation*}
\tilde{E}_i (0) = \chi_{T-\ln 3} E_i
\end{equation*}
because $\chi_{T-\ln 3}\chi_T = \chi_{T-\ln 3}$ ($s+\ln 3 \leq \ln 5 \Rightarrow s \leq \ln 4$). For $\tau=0 \dots t$, let
\begin{align*}
A(\tau) &= \chi_{T-\ln 3} e^{\frac{ i  \tau}{h}(P-W)} \chi_{T+ t} E_i \\
\frac{ d }{ d  \tau}A &= \chi_{T-\ln 3} e^{\frac{ i  \tau}{h}(P-W)} \frac{ i }{h} [ P, \chi_{T+t}] E_i.
\end{align*}
We want to use Egorov's lemma, first we need to localize the expression. Let $\epsilon >0$ be small enough, and take $f\in C^\infty_c(\R)$ so that $f$ is supported at distance less than $\epsilon$ of $1/2$ and equals $1$ near $1/2$. Let 
\begin{equation*}
F = \Op(f\circ p).
\end{equation*}
$F$ is a parametrix for $f(P)$, but we do not use that fact. We claim that
\begin{equation*}
(1-F) [P, \chi_{T+t}]E_i (s,.) = O_{L^2}(h^\infty)
\end{equation*}
First, remark that $f\circ p$ is indeed a symbol in the class $S^{-\infty}_0$. By ellipticity, we can solve
\begin{equation*}
1-F = \Op(r_n) (P-1/2)^n
\end{equation*}
for all $n\in \N$, with symbols $r_n$ in $S^{-2n}_0$. Observe
\begin{equation*}
(P-1/2) E_i = (W-1/2)E_i = \mathcal{O}(\eta h) E_i.
\end{equation*}
and 
\begin{equation*}
\left(P-1/2\right)^n [P,\chi] = \sum\limits_{k=0}^n \genfrac{(}{)}{0pt}{}{n}{k} P^{[k+1]}[\chi] \left(P - 1/2\right)^{n-k}
\end{equation*}
where $P^{[k]}[\chi]= [P, [ P, \dots, [P,\chi]\dots ]$ with $k$ occurences of $P$. From the proof of lemma \ref{lemma:approx_by_lagrangian_1}, we know that the $L^2$ norm of $E(s,.)$ restricted to any compact set is $\mathcal{O}(1/\eta)$. Now, since $r_n \in S^{-2n}_0$, $\Op(r_n)P^k[\chi]$ --- $k \leq n+1$ --- is bounded on $L^2$ with norm $h^k$, and is compactly supported; the claim follows since $\eta$ is bounded.

Now, we have localized our formulae in the momentum variable :
\begin{equation}\label{eq:1}
\frac{ d }{ d l}A = \chi_{T-\ln 3} e^{\frac{ i  l}{h}(P-W)} \left( \frac{ i }{h} F [ P, \chi_{T+t}] E_i + O_{L^2}(h^\infty) \right).
\end{equation}

According to the support hypothesis we have made, we can pick a function $g\in C^\infty_c(\R)$ such $g\circ y_M\equiv 1$ on a neighbourhood of $\supp(\partial_y \chi_{T+t})$, and that for all $0< l < t$, $\varphi_{-l}(\supp(f\circ p \times g\circ y_M))$ does not intersect the $\delta$-neighbourhood of $\supp(\chi_{T-\ln 3})$ where $\delta$ is some positive number. 

We can insert $1=g + 1-g$ in \eqref{eq:1} between $F$ and $[P, \chi_{T+t}]$. Now, remark \ref{remark:support_Egorov_lemma} gives that for $\epsilon>0$ small enough,
\begin{equation*}
\chi_{T-\ln 3} e^{\frac{ i  l}{h}(P-W)} Fg = e^{-\eta l} O_{L^2\to L^2}( (|t| h e^{2\lambda |t|})^\infty)
\end{equation*}

Since $\| [P,\chi_{T+t}] E_i \|_{L^2}$ is bounded by some finite power of $h$, we can conclude.
\end{proof}

We deduce the following lemma :
\begin{lemma}
For $\epsilon >0$ small enough, there is a symbol $\sigma_\epsilon$ that is supported at distance $\leq \epsilon$ of the energy shell $\{ p= 1/2\}$, and coincides with $\sigma$ on the neighbourhood $\{ 1/2 - \epsilon/2 \leq p \leq 1/2 + \epsilon/2\}$, such that
\begin{equation*}
\langle \Op(\sigma) E_i ,  E_j \rangle = \langle \Op(\sigma_\epsilon) \tilde{E}_i^0, \tilde{E}_j^0\rangle + \mathcal{O}\left( \frac{e^{-\eta t}}{\eta^2}\right) + \mathcal{O}((|t| he^{2\lambda |t|})^\infty) + \mathcal{O}(h^{\infty}).
\end{equation*}
\end{lemma}

\begin{proof}
We claim that the quantity in the LHS of the equation is well defined. We only have to prove that $\Op^\delta(\sigma):=y_M^\delta \Op(\sigma) y_M^\delta$ is bounded on $L^2$ for some $\delta > 0$ since then $y^{-\delta}E_i\in L^2(M)$. It suffices to prove it in the cusps. A simple computation shows that in $Z_\Lambda$, with $a\in S^n_\rho(Z_\Lambda)$, $y^\delta \Op(a) y^\delta= \Op(y^{2\delta} a)_\zeta$ with $\zeta(x)=4^\delta (x^2 + 4)^{-\delta}$. Since $\sigma$ is compactly supported, $y_M^{2\delta}\sigma$ still is a symbol, and $\Op^\delta(\sigma)$ is bounded on $L^2$.

From the pseudo-locality properties of $\Op(\sigma)$, and the bound $\| y^{-\epsilon} E_i \|_{L^2} = \mathcal{O}(1/\eta)$, we know that
\begin{equation*}
\langle \Op(\sigma) E_i ,  E_j \rangle = \langle \Op(\sigma) \chi_{T-\ln 3} E_i , \chi_{T-\ln 3}  E_j \rangle + \mathcal{O}(h^\infty).
\end{equation*}
We use the same trick as in the previous proof : we introduce $1=F + (1-F)$ with $F=\Op(f\circ p)$, where $f$ is smooth, supported in $[1/2 - \epsilon, 1/2 + \epsilon]$, and equals $1$ on $[1/2 - \epsilon/2, 1/2 + \epsilon/2]$. Then for the same reasons as above
\begin{equation*}
\Op(\sigma) \chi_{T-\ln 3} E_i = \Op(\sigma)F \chi_{T-\ln 3} E_i + \mathcal{O}_{L^2}(h^{\infty}).
\end{equation*}
But then $\Op(\sigma)F = \Op(\sigma_\epsilon) + R$ where $R$ is a negligible operator and $\sigma_\epsilon$ is as announced. From there :
\begin{equation*}
\begin{split}
\left| \langle \Op(\sigma)  E_i ,  E_j \rangle - \langle \Op(\sigma_\epsilon) \tilde{E}_i^0, \tilde{E}_j^0\rangle \right| & \leq  \langle \Op(\sigma_\epsilon) (\chi_{T-\ln3} E_i - \tilde{E}_i^0), \chi_{T-\ln 3}E_j \rangle \\
	& {}+ \langle \Op(\sigma_\epsilon) \chi_{T-\ln 3} E_i, \chi_{T-\ln 3} E_j - \tilde{E}_j^0 \rangle \\
	& {}+ \langle \Op(\sigma_\epsilon) (\chi_{T-\ln 3}E_i - \tilde{E}_i^0), \chi_{T-\ln3} E_j - \tilde{E}_j^0 \rangle \\
	& {}+ \mathcal{O}(h^\infty) 
\end{split}
\end{equation*}
We can conclude using the previous lemma, and :
\begin{equation*}
\| \chi_{T-\ln 3} E_i \|_{L^2} \leq C + \| (-\Delta - s(d-s))^{-1} [\Delta, \chi_i] E^0 \|_{L^2} \leq C(1+ \frac{1}{\eta})
\end{equation*}
\end{proof}

From now on, we choose a small enough $\epsilon>0$. We write :
\begin{equation*}
\langle \Op(\sigma_\epsilon) \tilde{E}_i^0, \tilde{E}_j^0 \rangle = e^{-\frac{ i  t W}{h}+ \frac{ i  t \overline{W}}{h} } \langle A \chi_{T+t} \chi_i E^0, \chi_{T+t}\chi_j E^0 \rangle
\end{equation*}
where, again with the notation $U(t) = e^{-\frac{i t P}{h}}$, 
\begin{equation*}
A = U(t) \chi_{T-\ln 3} \Op(\sigma_\epsilon) \chi_{T-\ln 3} U(-t).
\end{equation*}
Here again, Egorov's lemma gives 
\begin{equation*}
A = \Op(\sigma_\epsilon\circ \varphi_{-t}) + O_{L^2 \to L^2} ( h |t| e^{2\lambda |t|}).
\end{equation*}
Actually, when $i\neq j$, $\chi_{T+t} \chi_i E^0$ and $\chi_{T+t} \chi_j E^0$ have a distinct support, so that remark \ref{remark:support_Egorov_lemma} implies that when $i\neq j$, 
\begin{equation}\label{eq:main_estimate_off_diagonal}
\eta \langle \Op(\sigma)  E_i ,  E_j \rangle = \mathcal{O}\left( \frac{e^{-\eta t}}{\eta} \right) + \mathcal{O}( \eta (|t| h e^{2\lambda |t|})^\infty) + \mathcal{O}(\eta h^\infty).
\end{equation}
Now, we assume that $i=j$, unless specifically stated. We denote $\sigma_{\epsilon,t} = \sigma_\epsilon \circ \varphi_{-t}$. We claim that when $\eta$ remains bounded and $\eta \times t \to \infty$, there are constants $C_1$ and $C_2$ (depending on $T$) such that
\begin{equation}\label{eq:principal_constant}
\frac{C_2}{\eta} \leq e^{-2\eta t} \| \chi_{T+t} \chi_i E^0 \|^2_{L^2} \leq  \frac{C_2}{\eta}.
\end{equation}
Indeed
\begin{align*}
e^{-2\eta t} \| \chi_{T+t} \chi_i E^0 \|^2_{L^2} &= \int_{y>a}  \chi_{T+t}(y)^2 \chi(y)^2 y^{1+2\eta} e^{-2\eta t} \frac{ d y}{y^2}  \\
			&= \int_{s> \ln a}  e^{2\eta s} \tilde{\chi}(s-T -t -\ln a)^2 \chi(e^s)^2 e^{-2\eta t} d s  \\
e^{-2\eta t} \| \chi_{T+t} \chi_i E^0 \|^2_{L^2}			&\leq  \int_{\ln 2a}^{\ln(5a) + T + t} e^{2\eta (s-t)} = \frac{1}{2\eta}[e^{2\eta(\ln(5a) + T)} - e^{2\eta(\ln(2a) - t)} ]  \\
			& \leq \frac{C_2}{\eta}(1+o(1))  \\
e^{-2\eta t} \| \chi_{T+t} \chi_i E^0 \|^2_{L^2}			&\geq  \int_{\ln 3a}^{\ln(4a) + T + t} e^{2\eta (s-t)} = \frac{1}{2\eta}[e^{2\eta(\ln(4a) + T)} - e^{2\eta(\ln(3a) - t)} ] \\
			& \geq \frac{C_1}{\eta}(1+o(1)).
\end{align*}

Hence, when $\eta \times t \to +\infty$, and $\lambda > \lambda_{max}$,
\begin{equation*}
\eta \langle \Op(\sigma) E_i, E_i \rangle = \eta e^{-2\eta t} \langle \Op(\sigma_t) \chi_{T+t} \chi_i E^0, \chi_{T+t} \chi_i E^0 \rangle + O\left(\frac{e^{-2\eta t}}{\eta} + (h|t|e^{2\lambda |t|})^\infty + h |t|e^{2\lambda |t|}\right).
\end{equation*}
Letting $t= t_0 |\log h |/(2\lambda)$, where $0< t_0 < 1$, and assuming 
\begin{equation*}
\eta \geq C_{\lambda} \frac{\log |\log h|}{|\log h|}
\end{equation*}
with $C_\lambda > \lambda / t_0 > \lambda_{max}(1/2) $, we find
\begin{equation}\label{eq:conclusion_lagrangian_state_reduction}
\eta \langle \Op(\sigma) E_i, E_i \rangle = \eta e^{-2\eta t} \langle \Op(\sigma_{\epsilon,t}) \chi_{T+t} \chi_i E^0, \chi_{T+t} \chi_i E^0 \rangle + o_{h\to 0}(1)
\end{equation}
For $i\neq j$, equation \ref{eq:main_estimate_off_diagonal} gives
\begin{equation}
\eta \langle \Op(\sigma) E_i, E_j \rangle = o_{h\to 0}(1).
\end{equation}

\subsubsection{Stationary phase computations}

The idea behind the proof here is that $\chi_{T+t}\chi_i E^0$ is a lagrangian state, thus mapped to another lagrangian state by $\Op(\sigma_{\epsilon,t})$ which is a pseudo-differential operator. 

\begin{lemma}\label{lemma:approx_dynamical_incoming_measure} Assume $t_0$, $\lambda$ and $\eta$ satisfy the above conditions. Then,
\begin{equation*}
\begin{split}
\eta e^{-2\eta t} \langle \Op(\sigma_{\epsilon,t}) &\chi_{T+t} \chi_i E^0, \chi_{T+t} \chi_i E^0 \rangle = \\
& \left[2 \pi a^{2\eta}\eta e^{-2\eta t} \int  d \theta  d \tau e^{2\eta \tau} [\chi_{T+t}\chi]^2(ae^\tau) \sigma_{\epsilon,t-\tau}(a, \theta, \frac{1}{a}, 0)\right] + \mathcal{O}(h^{1-t_0})
\end{split}
\end{equation*}
\end{lemma}

\begin{proof} This computation only takes place in cusp $Z_i$, and we forget the dependence in $i$ until the end of the proof of this lemma.

First, we can eliminate the integration in the $\theta'$ and $J$ variable in the LHS because of the following fact. When $\varsigma$ is tempered, $\Lambda$ periodic in the first variable,
\begin{align*}
\int_D  \left(\int_{\R^{2d}}  \varsigma\left(\frac{\theta + \theta'}{2}, hJ\right) e^{ i  \langle \theta - \theta',J\rangle}d \theta'  d J\right) d \theta  & = \int_D  \sum_{J\in \Lambda'} \hat{\varsigma}(J/2, hJ/2) e^{ i  \langle \theta, J\rangle}d \theta  \\
&=\int_D \varsigma(\theta, 0) d \theta 
\end{align*}
where $\hat{\varsigma}$ was the discrete Fourier transform in the first variable. Hence, the quantity in the LHS in the lemma is the integral over $\theta \in D$ of the following expression :
\begin{equation}\label{eq:formula_cusp_42}
h^{-1}\eta e^{-2\eta t} \int y^{s-1}y'^{\overline{s}-1} e^{ i  (y-y')Y/h} [\chi_{T+t}\chi](y) [\chi_{T+t}\chi](y') \sigma_{\epsilon,t}\left( \frac{y+y'}{2},\theta, Y, 0 \right) d y  d y'  d Y .
\end{equation}
We want to use the fact that if $\varsigma$ is a symbol in some $S^n(Z)$, not depending on $\theta$ nor on $J$, then the function $\tilde{\varsigma}(s,v)= \varsigma(e^s, e^{-s} v)$ is a symbol in the usual Kohn-Nirenberg sense, in $S^n(\R)$ --- notation of definition \ref{def:symbols_manifold}. Remark that the behavior is not so clear in the $\theta$ variable, for which periodicity and rescaling are not compatible.

We introduce the following rescalings : $y=ae^\tau$, $y' = y (1+u) $, $Y=(1+v)/y$. Up to a factor $h^{-1}\eta a^{2\eta} e^{-2\eta(t-\tau)} \chi_{T+t}\chi(ae^{\tau})$, the expression in \eqref{eq:formula_cusp_42} is the integral over $\tau \in \R$ of
\begin{equation*}
\int  (1+u)^{\eta - 1/2} e^{ i  (\log(1 + u) - u (1+v))/h}[\chi_{T+t}\chi](a e^\tau (1+u)) \sigma_{\epsilon,t}\left( a e^\tau(1+ \frac{u}{2}),\theta, \frac{1+v}{a e^\tau}, 0 \right) d u  d v  .
\end{equation*}
Remark that this integral vanishes when $\tau \notin [ \ln 2, T+ t + \ln 5]$, and write 
\begin{equation*}
\sigma_{\epsilon,t}\left( a e^\tau(1+ \frac{u}{2}),\theta, \frac{1+v}{a e^\tau}, 0 \right) = \sigma_{\epsilon,t-\tau}\left( a (1+ \frac{u}{2}),\theta, \frac{1}{a}(1+v), 0 \right).
\end{equation*}
Then, introduce a cutoff $\varrho(u)$, supported around $0$, and $1= \varrho + 1-\varrho$ to separate the integral into two parts (I) and (II). 

Let us examine first (II) which is not stationnary, and supported for $|u|> \delta$. We insert $1 = u^N /u^N$ and integrate by parts in $v$. We take the $L^1$ bound, considering that $\sigma_{\epsilon,t}$ is supported in $\{ p \in [1/2 - \epsilon, 1/2 + \epsilon]\}$, and using symbol estimates on $\sigma_{\epsilon,t}$. It gives
\begin{equation*}
\begin{split}
|\mathrm{(II)}| \leq C_N h^N e^{N\lambda(t-\tau)}\int  (1-\varrho(u))& (1+u)^{\eta- 1/2} \frac{(1+ u/2)^N}{u^N}  [\chi_{T+t}\chi](a e^\tau (1+u))  \\
		&\mathbb{1}\left[(1+u/2)|1+v| \leq \sqrt{1+2\epsilon}\right] d u  d v .
\end{split}
\end{equation*}
after some rescaling in the $v$ variable, and considering 
\begin{equation*}
[\chi_{T+t}\chi](a e^\tau (1+u)) \leq \mathbb{1}(-1 \leq u \leq +\infty),
\end{equation*}
this is bounded (uniformly in $\theta$ and $\tau$) by
\begin{equation*}
C h^{N(1-2\rho)} \int_{-1}^{+\infty}  d u \frac{(1-\varrho) (1+u)^{\eta- 1/2} \big(1+ u/2\big)^{N-1}}{u^N} = \mathcal{O}\left(h^{2-t_0}\right).
\end{equation*}

Part (I) of the integral supported around $u=0$ is an oscillatory integral that can directly be estimated. Indeed, on that domain, the phase function satisfies symbolic estimates and has only one critical point $(u,v)=(0,0)$, where it is $- uv + \mathcal{O}(u^3,v^3)$. Further consider that the function under the integral is smooth and uniformly compactly supported in $v$. When we differentiate it in $v$, we lose a $\mathcal{O}(e^{\lambda |t-\tau| })$ constant. When differentiating in $u$, either we differentiate $\sigma_{\epsilon, t-\tau}$, losing again a $\mathcal{O}(e^{\lambda |t-\tau| })$ constant, or we differentiate $\rho(u)(1+u)^{\eta -1/2} \chi_{T+t}\chi$. We chose --- recall \eqref{def:sharp_cutoffs} --- the cutoffs $\chi_{T+t}$ and $\chi$ exactly so that we lose only $\mathcal{O}(1)$ constants by doing so.

The basic stationnary phase theorem --- see theorem 7.7.5 in \cite{MR1996773} --- in the plane applies and we find
\begin{equation*}
\mathrm{(I)} = 2\pi h [\chi_{T+t}\chi](a e^\tau ) \sigma_{\epsilon,t-\tau}\left( a,\theta, \frac{1}{a}, 0 \right) + \mathcal{O}(h^{2-t_0}),
\end{equation*}
uniformly in variables $\theta$ and $\tau$.  \\

Recall we are to integrate (I)+(II) in $\theta$ and $\tau$ with a prefactor $h^{-1}\eta a^{2\eta} e^{-2\eta(t-\tau)} \chi_{T+t}\chi(ae^{\tau})$. But 
\begin{equation*}
\int  d \theta  d \tau \eta a^{2\eta} e^{-2\eta(t-\tau)} \chi_{T+t}\chi(ae^{\tau}) = \mathcal{O}(1).
\end{equation*}
and that estimate ends the proof.
\end{proof}

\subsubsection{Dynamical properties and conclusion}

Recall that whenever $\varsigma$ is a compactly supported continuous function on $T^\ast M$, in the coordinates of cusp $Z_i$,
\begin{equation*}
\mu^+_{i,\nu}(\varsigma) = 2\nu a^{2\nu}\int_{t\in \R} e^{-2\nu t} \varsigma \circ \varphi_{-t}\big(a, \theta, \frac{1}{a}, 0\big)  d \theta.
\end{equation*}
When $\eta(h)\to \nu>0$, $\eta$ remains bounded, and we can directly apply lemma \ref{lemma:approx_dynamical_incoming_measure} and equation \eqref{eq:conclusion_lagrangian_state_reduction}. Letting $h\to 0$, we find 
\begin{equation*}
\eta \langle \Op(\sigma) E_i(s), E_i(s)\rangle \to \pi \mu^+_{i,\nu}(\sigma_\epsilon).
\end{equation*}
Since $\mu^+_i$ is supported on $S^\ast M$
Actually, when $\eta \to 0$, slowly enough, lemma \ref{lemma:approx_dynamical_incoming_measure} and equation \ref{eq:conclusion_lagrangian_state_reduction} also imply that 
\begin{equation*}
\begin{split}
|\eta \langle \Op(\sigma) E_i(s), E_i(s)\rangle - \pi \mu^+_{i,\eta(h)}(\sigma)| &= \mathcal{O}(h^{1-t_0}) + \mathcal{O}(\|\sigma\|_{L^\infty} e^{-2\eta(h) t(h)}) \\
	& = \mathcal{O}(|\log h|^{-t_0}).
\end{split}
\end{equation*}

The proof of theorem \ref{theorem:Extension_Semyon} will therefore be complete if we can prove 
\begin{lemma}\label{lemma:convergence_measure}
Assume $M$ has strictly negative curvature. For all $\sigma\in C^0_c(T^\ast M)$, for all $i=1\dots k$, as $\nu \to 0^+$,
\begin{equation*}
\mu^\pm_{i,\nu}(\sigma) \to \int \sigma  d \mathscr{L}_1
\end{equation*}
where $\mathscr{L}_1$ is the normalized Liouville measure on the unit cotangent bundle of $M$.
\end{lemma}

It is as far as we know an open question as to whether the Liouville measure is a Gibbs measure in such a cusp-manifold --- that is to say, whether a Ruelle inequality holds. If it were, we could apply directly theorem 3 in \cite{MR1910932}. However, mimicking the proof therein and using the classical Hopf argument, we are able to conclude. Observe that replacing the hypothesis of negative curvature by \emph{ergodic}, or even \emph{mixing}, we are not able to prove that conclusion still holds: we really use the stable and unstable foliations, and the fact that they are absolutely continuous. 

In this part of the proof, it is easier to consider only $\mu^-$, that is supported on incoming horospheres.

\begin{proof} 
From now on, we work on the unit cotangent sphere since both $\mu^-_{\nu, i}$ and $\mathscr{L}_1$ are supported on $S^\ast M$.

Take $\varepsilon >0$. Since $\sigma$ is compactly supported, we can find a $\delta >0$ such that $|\sigma(\xi)-\sigma(\xi')|< \varepsilon$ whenever $|\xi-\xi'| < \delta$. The measures $\mu^-_{i,\nu}$ are obtained by propagating an incoming horocycle. Following the idea of proof in \cite{MR1910932}, we want to thicken the horocycle. Denote by $H_{i,a}$ the incoming horosphere at height $a$ in cusp $Z_i$, and consider the set
\begin{equation*}
\Omega_{i,a,C} := \bigcup_{\xi \in H_{i,a}} B(\xi, C, W^{s}(\xi)),
\end{equation*}
where $B(\xi, C, W^{s}(\xi))$ is a ball of radius $C>0$ in the local weak stable leaf of $\xi$, centered at 	$\xi$. For $C>0$ small enough, in $\Omega_{i,a,C}$, $H_{i,a}$ is a local section of the weak-stable foliation of $T^{\ast,1} M$, with projection $\pi^{su}$. From the contraction properties on the weak-stable foliation, for some constant $C>0$, on $\Omega_{i,a,C\delta}$, $|\sigma\circ \varphi_t - \sigma\circ \varphi_t\circ \pi| \leq \varepsilon$. 

From theorem 7.6 in \cite{2012arXiv1211.6242P}, there is a locally bounded measurable density $\rho$ such that in $\Omega_{i,a,C\delta}$,
\begin{equation*}
 d \mathscr{L}_1 = \rho(\xi)  d \mathrm{vol}_{W^{s}(\pi^{su} \xi)}  d \mathrm{vol}_{H_{i,a}}
\end{equation*} 
The measure on $H_{i,a}$ being $d\theta$ of mass $1$. We let $g$ vanish out of $\Omega_{i,a,C\delta}$ and on $\Omega_{i,a,C\delta}$,
\begin{equation*}
g(\xi) := \frac{1}{\rho(\xi) \mathrm{vol}(B(\pi^{su}\xi, C, W^{s}(\pi^{su}\xi)))}.
\end{equation*}
Then, $g$ is in $L^1(S^\ast M)$, and $\| g \|_{L^1(S^\ast M)} = \mathrm{vol}(H_{i,a})= 1$. Hence, 
\begin{equation*}
\left| \int_{H_{i,a}} \sigma\circ \varphi_t  d \theta - \int_{\Omega_{i,a,C\delta}} \sigma\circ \varphi_t \times g d \mathscr{L}_1\right|.
\end{equation*}

Consider that in the definition of $\mu^+_{\nu,i}$, since $\sigma$ is compactly supported, we can integrate in $t$ for only $t\in [-T, + \infty[$. Additionally, the prefactor $a^{2\eta}$ tends to $1$, and the part $t\in [-T, 0]$ will not contribute, so we write
\begin{equation*}
\mu^-_{\nu,i} (\sigma) = \mathcal{O}(\nu T  \| \sigma \|_{\infty}) + 2\nu \int_0^\infty  d t e^{-2\nu t} \int \sigma \circ \varphi_t  \times g  d \mathscr{L}_1  + \mathcal{O}(\varepsilon)
\end{equation*}
Using the Hopf argument (as in \cite{MR2261076}), and theorem 7.6 from \cite{2012arXiv1211.6242P} again, one can see that the geodesic flow is mixing for the Liouville measure. Actually, it suffices for it to be \emph{ergodic}. Indeed, 
\begin{equation*}
2\nu \int_0^\infty  d t e^{-2\nu t} \int \sigma \circ \varphi_t \times g  d \mathscr{L}_1 = \int_{\R^+} t e^{-t} \int_{T^{\ast,1}M} g(\xi) F\left(\frac{t}{2\nu},\xi\right) d t  d \mathscr{L}_1(\xi),
\end{equation*}
where $F(t,\xi)$ is the Birkhoff average of $\sigma$ for a time $t$ along the trajectory of $\xi$. Since $g$ is $L^1$, and $\sigma$ is bounded, by dominated convergence and ergodicity, the limit of this when $\nu \to 0$ is $\|g\|_{L^1}\mathscr{L}_1(\sigma)$. For all $\varepsilon>0$, we find for any limit value $\overline{\sigma}$ of $\mu^-_{i,\nu}(\sigma)$,
\begin{equation*}
\left| \overline{\sigma} - \mathscr{L}(\sigma) \right| = \mathcal{O}(\varepsilon).
\end{equation*}
letting $\varepsilon\to 0$ yields the desired result.
\end{proof}

\appendix

\section{Functionnal spaces in a cusp}
\label{appendix:functionnal_spaces_cusp}

First, let us recall some definitions on covariant derivatives. If $S$ is a tensor on a riemannian manifold, one defines its covariant derivative in the following way:
\begin{equation*}
(\nabla_X S)(Y_1, \dots, Y_n) := X\left( S(Y_1, \dots, Y_n) \right) - \sum_i S(Y_1, \dots, \nabla_X Y_i, \dots, Y_n)
\end{equation*}
In particular, when $f$ is a function on a riemannian manifold $N$, one defines a family of tensors $\nabla^n f$ in the following way. 
\begin{equation*}
\nabla f : X \mapsto X(f) \quad \text{ and } \quad \nabla^{n+1}_{X_0, \dots, X_n}f := (\nabla_{X_0} \nabla^{n})_{X_1, \dots X_n} f.
\end{equation*}
We also define it for vectors --- which are $(0,1)$ tensors:
\begin{equation*}
\nabla Z : X \mapsto \nabla_X Z \quad \text{ and } \quad \nabla^{n+1}_{X_0, \dots, X_n}Z := (\nabla_{X_0} \nabla^{n})_{X_1, \dots X_n} Z.
\end{equation*}
This enables us to define, for $x\in N$
\begin{equation*}
\| \nabla^n f \|(x) = \sup_{X_1, \dots, X_n \in T_x N} \frac{ |\nabla^n f(X_1, \dots, X_n)|}{\| X_1\| \dots \| X_n \|}
\end{equation*} 
Then, the space $\mathscr{C}^n(N)$ is the set of functions on $N$ that are $C^n$, and such that 
\begin{equation*}
\| f \|_{\mathscr{C}^n(N)} : = \sum_{k=0}^n \sup_{x\in N} \| \nabla^k f \|(x)  < \infty.
\end{equation*}

Now, we turn to Sobolev spaces. When $N$ is complete, $L^2(N)$ is a Hilbert space. For $n\geq 0$ an integer, one defines the norm
\begin{equation*}
\| f \|^2_{H^n(Z)} := \sum_{k\leq n}  \left\| \|\nabla^k f\|(x) \right\|^2_{L^2(dx)}
\end{equation*}
The Sobolev space $H^n(N)$ of order $n$ is the completion of $C^\infty(N)$ for this norm. If $N$ has no boundary, then $H^{-n}(N)$ is defined as the dual of $H^n(N)$.

Using the Lax-Milgram theorem, exactly as for the Laplacian on $\R^n$, one proves that for any $\epsilon>0$, $-\Delta+ \epsilon$ is invertible on $H^1(N)$ with values in $H^{-1}(N)$. Since it is also positive, one can use the spectral theorem to define $(-\Delta + 1)^s$ for any $s\in \R$. One observes that $\| . \|_{H^1(N)}$ and $\| (-\Delta + 1)^{1/2} . \|$ are equivalent norms on $H^1(N)$.

The cusp $Z$ is complete, so the above apply. Now, one can compute the following :
\begin{equation}
\begin{split}\label{eq:Christoffel_symbols}
\nabla_{X_y} X_y &= 0\\
\nabla_{X_y} X_{\theta_i} &= 0\\
\nabla_{X_{\theta_i}} X_y &= - X_{\theta_i}\\
\nabla_{X_{\theta_i}} X_{\theta_j} &= \delta_{ij} X_y 
\end{split}
\end{equation}
From this and the definition of $\nabla^n$, if $\alpha$ is a space-index of length $n$, we find
\begin{equation*}
\nabla^n f(X_\alpha) = X_\alpha f + \sum_{\beta} \pm X_\beta f
\end{equation*}
where $\beta$ are other space-indices, of length $< n$. Whence by induction on $n\geq 0$ we find
\begin{equation}\label{eq:equivalent_norms_C^n}
\| f \|_{\mathscr{C}^n(Z)} \text{ is equivalent to } \sum_{|\alpha| \leq n} \| X_\alpha f \|_{L^\infty(Z)}.
\end{equation}
and
\begin{equation}\label{eq:equivalent_norms_Sobolev_1}
\| f \|_{H^n(Z)} \text{ is equivalent to } \sum_{|\alpha| \leq n} \| X_\alpha f \|_{L^2(Z)}.
\end{equation}

Now, we define, for $s$ a real number
\begin{equation*}
\| f \|_s := \| (-\Delta + 1)^s f \|_{L^2}
\end{equation*}
We want to show that the completion of $C^\infty(Z)$ \emph{is} the sobolev space $H^s( Z)$ for integer $s$, and that then, $\|. \|_s$ is equivalent to $\| . \|_{H^s(Z)}$. Such a result is deduced of an elliptic estimate similar to that in pp 358 in \cite{MR2744150}. Actually, the proof therein adapts to a cusp if one defines the slope operators $D_{j,h}$ in the following way
\begin{equation*}
D_{j,h}f(x) = \frac{1}{h}\left( f(x+ h X_j)-f(x) \right), \quad j = y, \theta_1, \dots, \theta_d.
\end{equation*}

Then, using $P=-h^2 \Delta/2$, we also define the semi-classical Sobolev norms :
\begin{equation*}
\| f \|_{s,h} := \| (P+1)^{s/2}f\|_{L^2(Z)}.
\end{equation*}
One gets for some constant $C>0$
\begin{equation}\label{eq:equivalent-norms}
\frac{1}{C}h^{s^+}\| f \|_s \leq \|f \|_{s,h} \leq C h^{-s^-}\| f \|_s.
\end{equation}
where  $s^+$ and $s^-$ are the positive and negative part of $s$.

To finish this section, we define the non-integer Sobolev spaces using complex interpolation --- as in pp 321 from \cite{MR2744150}.

\section{Estimating the derivatives of a flow on a Riemannian manifold}

The following proposition should be classical, but for lack of a reference, we enclose a proof.
\begin{prop}\label{prop:estimating_derivatives_propagation_flow}
Let $\varphi^t$ be a flow in a manifold $N$, such that all the covariant derivatives of the vector field $V$ of the flow and of the curvature tensor of $N$ are bounded. Assume also that the maximal Lyapunov exponent $\lambda_0$ of $\varphi^t$ --- as defined in \ref{def:max_lyapunov} --- is \emph{finite}. Then for all $\lambda > \lambda_0$, there are constants $C_n >0$, such that for $f \in \mathscr{C}^n (N)$, for $t \in \R$,  
\begin{equation*}
\| f \circ \varphi_t \|_{\mathscr{C}^n(N)} \leq C_n e^{n \lambda |t|} \|f\|_{\mathscr{C}^n(N)}
\end{equation*}
\end{prop}
The proof is inspired by \cite{MR3215926} ($N$ is a convex co-compact hyperbolic surface), which itself comes from \cite{MR1882134} ($N= \R^n$). In the usual proofs of this type of result, at some point, one uses coordinates to transport the problem to $\R^n$. When $N$ is compact, this is reasonnable because all metrics on $N$ are equivalent. When $N$ is non compact, it is probably possible to take a similar approach. However, one would have be careful and take coordinate charts with derivatives nicely bounded. We chose to avoid taking coordinates altogether, and give an intrinsic formulation of the proof, hence the appearance of many tensors.

The main idea of the proof is to avoid estimating higher derivatives of the flow, and replace them by higher derivatives of the vector field of the flow.
\begin{proof}
We want to compute 
\begin{equation*}
\nabla_{X_1, \dots X_n} (f\circ \varphi^{t}).
\end{equation*}
We are going to compare this with
\begin{equation*}
\nabla_{\varphi^t_\ast X_1, \dots \varphi^t_\ast X_n} f.
\end{equation*}
In the first expression, there are \emph{a priori}, higher derivatives of the flow, while the second one only contains first order derivatives that are much easier to estimate. Let $z\in N$, and $X_1, \dots, X_n \in T_z N$. Let
\begin{equation*}
W_t^n(X_1, \dots X_n)f := \left[(\varphi^t)^\ast \left(\nabla_{\varphi^t_\ast X_1, \dots, \varphi^t_\ast X_n} \right)\right]f
\end{equation*}
That is :
\begin{equation*}
\left[W_t^n(X_1, \dots, X_n) f\right]\circ\varphi_{-t} = \left(\nabla_{\varphi^t_\ast X_1, \dots, \varphi^t_\ast X_n} \right)(f\circ \varphi_{-t}).
\end{equation*}
From the definition, we see that $W_t^n f$ is a tensor. We observe that
\begin{equation}\label{eq:covariant_derivative_twisted_induction}
W_t^n = \nabla W_t^{n-1} + \sum_{i=2,\dots n} W_t^{n-1}(X_2, \dots, \nabla_{X_1}X_i - (\varphi^t)^\ast (\nabla_{\varphi^t_\ast X_1} \varphi^t_\ast X_i), \dots, X_n).
\end{equation}
One can compute $(\varphi^t)^\ast (\nabla_{\varphi^t_\ast X_1} \varphi^t_\ast X_i)$. Indeed, consider the fact
\begin{equation*}
\partial_t (\varphi^t)^\ast X(t) = (\varphi^t)^\ast[ V, X(t)] + (\varphi^t)^\ast \partial_t X
\end{equation*}
We deduce that
\begin{align*}
\partial_t (\varphi^t)^\ast (\nabla_{\varphi^t_\ast X} \varphi^t_\ast Y) &=  (\varphi^t)^\ast\left( [ V, \nabla_{\varphi^t_\ast X} \varphi^t_\ast Y] - \nabla_{[V,\varphi^t_\ast X]} \varphi^t_\ast Y - \nabla_{\varphi^t_\ast X} [V,\varphi^t_\ast  Y]\right).\\
\intertext{That is}
\partial_t (\varphi^t)^\ast (\nabla_{\varphi^t_\ast X} \varphi^t_\ast Y) &=  (\varphi^t)^\ast Z(\varphi^t_\ast X,\varphi^t_\ast Y) \\
\intertext{with}
Z(\varphi^t_\ast X,\varphi^t_\ast Y) & = \nabla^2_{\varphi^t_\ast X \varphi^t_\ast Y}V + R_\nabla(V, \varphi^t_\ast X)\varphi^t_\ast Y,
\end{align*}
where $R_\nabla$ is the curvature tensor. So,
\begin{equation}\label{eq:covariant_vector_vector}
\nabla_X Y - (\varphi^t)^\ast (\nabla_{\varphi^t_\ast X} \varphi^t_\ast Y) = - \int_0^t (\varphi^s)^\ast Z(\varphi^s_\ast X, \varphi^s_\ast Y) ds
\end{equation}
This equation we found in \cite{MR3215926}, and the rest of the proof is devoted to proving similar formulae for higher order derivatives. Let us call the tensor in the RHS $L_t(X,Y)$. We can already compute explicitly
\begin{equation*}
W^1_t f = X_1 f \quad W^2_t f = \nabla_{X_1, X_2} f + L_t(X_1, X_2)f.
\end{equation*}
Now, we introduce a class of vector-valued tensors $\mathscr{T}$. Elements of $\mathscr{T}$ depend on two time-parameters $s, t$. First, the identity is in $\mathscr{T}$. Second, if $T_1(s,t), \dots, T_k(s,t)$ are in $\mathscr{T}$ with $k\geq 2$, and if $R$ is a smooth (vector-valued) $k$-tensor with all its covariant derivatives bounded, then
\begin{equation}\label{eq:form_tensors_recurrent}
T^R_{T_1, \dots, T_k}(s,t) := \int_s^t (\varphi^u)^\ast R(\varphi^u_\ast T_1(u, t), \dots, \varphi^u_\ast T_k(u, t))du
\end{equation}
is also in $\mathscr{T}$. We require that $\mathscr{T}$ is the vector space generated by the above tensors. For example, $L_t$ is $T^{-Z}_{Id,Id}(0,t)$; we denote $L(s,t)=T^{-Z}_{Id,Id}(s,t)$. Then
\begin{lemma}
$W^n_t f$ can be written as a sum of terms
\begin{equation}\label{eq:terms_W^n_t}
\nabla^k f(T_1(0,t), \dots, T_k(0,t))
\end{equation}
where the $T_i$'s are in $\mathscr{T}$ of the correct order.
\end{lemma}

\begin{proof}
We have already checked it for $n=1$ and $n=2$. Actually, we will check that if $A_t$ is an $(n-1)$-tensorial operator of the type \eqref{eq:terms_W^n_t}, then
\begin{equation*}
\nabla A_t - \sum_{i=1, \dots, n} A_t(X_2, \dots, L_t(X_1, X_i), \dots, X_n)
\end{equation*}
is a sum of such operators (of orders $n$ and $n-1$). Let us observe that
\begin{equation*}
\begin{split}
\nabla \left[\nabla^k (T_1(0,t), \dots, T_k(0,t)) \right]&= \nabla^{k+1}(Id, T_1(0,t), \dots, T_k(0,t))  \\
		& - \sum_{i=1,\dots, k} \nabla^k (T_1(0,t), \dots,\nabla T_i(0,t), \dots, T_k(0,t)).
\end{split}
\end{equation*}
We deduce that it suffices to show that when $T\in \mathscr{T}$ is a $k$-tensor, 
\begin{equation*}
\begin{split}
T':=\nabla_X T(s,t)& + \int_0^s (\varphi^w)^\ast Z(\varphi^w_\ast X,\varphi^w_\ast T(s,t)) dw \\
	&+ \sum_{i=1}^k T(0,t)(X_1, \dots, L_t(X_0,X_i), \dots, X_k) 
\end{split}
\end{equation*}
is in $\mathscr{T}$. We prove this by induction on $k$. First, if $k=1$, $T$ is the identity, and
we find $T'(s,t) = L(s,t)$.

Assume we are done for all $k\leq n$. Then, let $T$ be a $n+1$ tensor in $\mathscr{T}$. By construction, it is a sum of terms as in \eqref{eq:form_tensors_recurrent}. Since the property we are trying to prove is stable by taking sums, assume there is only one term in the sum. The $T_i$'s all are of order $< n+1$, and we can compute, using \eqref{eq:covariant_vector_vector} in the first line
\begin{align*}
\begin{split}
\nabla_X T(s,t) 	&= \int_s^t  \int_0^u (\varphi^{u-w})^\ast (-Z)(\varphi^{u-w}_\ast X, \varphi^{-w}_\ast R(\varphi^u_\ast T_1(u, t), \dots, \varphi^u_\ast T_k(u, t))) dw du \\
				& \quad + \int_s^t (\varphi^u)^\ast \nabla_{\varphi^u_\ast X}\big( R(\varphi^u_\ast T_1(u, t), \dots, \varphi^u_\ast T_k(u, t))\big) du. 
\end{split}\\
\begin{split}
\nabla_X T(s,t) 	&= T^{-Z}_{Id, T}(s,t) + \int_0^s (\varphi^w)^\ast (-Z)(\varphi^w_\ast X,\varphi^w_\ast T(s,t)) dw  \\
				&\quad + \int_s^t (\varphi^u)^\ast (\nabla_{\varphi^u_\ast X} R)(\varphi^u_\ast T_k(u, t), \dots, \varphi^u_\ast T_k(u, t)) du. \\
				&\quad + \sum_{i=1}^k \int_s^t (\varphi^u)^\ast  R(\varphi^u_\ast T_1(u, t), \dots, \nabla_{\varphi^u_\ast X} \varphi^u_\ast T_i(u,t), \dots, \varphi^u_\ast T_k(u, t)) du.
\end{split}
\end{align*}
Hence we find
\begin{equation*}
\begin{split}
\nabla_X T(s,t) + &\int_0^s (\varphi^w)^\ast Z(\varphi^w_\ast X,\varphi^w_\ast T(s,t)) dw = T^{-Z}_{Id, T}(s,t) + T^{\nabla R}_{Id, T_1, \dots, T_k}(s,t)  \\
&+ \sum_{i=1}^k \int_s^t (\varphi^u)^\ast  R(\varphi^u_\ast T_1(u, t), \dots, \nabla_{\varphi^u_\ast X} \varphi^u_\ast T_i(u,t), \dots, \varphi^u_\ast T_k(u, t)) du
\end{split}
\end{equation*}
But we precisely have
\begin{equation*}
(\varphi^u)^\ast \nabla_{\varphi^u_\ast X} \varphi^u_\ast T_i(u,t) = \nabla_{X} T_i(u,t) + \int_0^u (\varphi^w)^\ast Z(\varphi^w_\ast X, \varphi^w_\ast T_i(u,t)) dw.
\end{equation*}
so we can use the induction hypothesis, and conclude.

\end{proof}

\begin{lemma}
When $T\in \mathscr{T}$ is a $n$-tensor, there is a constant $C>0$ such that whenever $0\leq s\leq t$, 
\begin{equation*}
\| \varphi^s_\ast T(s,t)((\varphi^t)^\ast X_1, \dots, (\varphi^t)^\ast X_n) \| \leq C e^{n \lambda (t-s)}\| X_1 \| \dots \| X_n\|.
\end{equation*}
\end{lemma}

\begin{proof}
We proceed by induction. First, for the identity, this is true because the maximal lyapunov exponent of the flow is bounded. Now, we assume it is true for all $k$-tensors in $\mathscr{T}$ with $k\leq n$, and let $T\in \mathscr{T}$ be a $n+1$ tensor.
\begin{equation*}
\| \varphi^s_\ast T(s,t) \| \leq \int_s^t e^{\lambda(u-s)} \prod_{i=1}^k \| \varphi^u_\ast T_i(u,t) \| du
\end{equation*}
If we use the induction hypothesis, we get
\begin{equation*}
\| \varphi^s_\ast T(s,t) \| \leq C \| X_1 \| \dots \| X_n\| \underset{\leq C e^{\lambda (n+1) (t-s)}}{\underbrace{\int_s^t e^{\lambda(u-s) + (n+1) \lambda(t-u)}du}}
\end{equation*}
\end{proof}

We conclude the proof by observing that
\begin{equation*}
\nabla^n (f\circ \varphi_{-t})(X_1, \dots, X_n) = W^n_t f((\varphi^t)^\ast X_1, \dots, (\varphi^t)^\ast X_n)
\end{equation*}

\end{proof}

\section{On the Sasaki metric}\label{appendix:Sasaki}

\subsection{The curvature tensor of a Sasaki metric}

There is a useful --- and easily accessible --- reference for the Sasaki metric on tangent spaces: \cite{MR1888866}. We are going to rely heavily on it to avoid introducing too much machinery. In the following paragraph, we retain the notations therein. We want to show that proposition \ref{prop:estimating_derivatives_propagation_flow} applies to the geodesic flow cusp surfaces. We prove
\begin{prop}
Assume that the curvature tensor of $M$ is bounded, and all its covariant derivatives also. For $R>0$, let $TM_R :=\{ v \in TM \; |\; \| v \|^2 \leq 2R\}$ be endowed with the Sasaki metric. Then
\begin{enumerate}
	\item The curvature tensor of $TM_R$, and all its derivatives are bounded.
	\item It is also the case for the vector of the geodesic flow\end{enumerate}
\end{prop}
Remark that when the curvature of $M$ is constant, the covariant derivative of the curvature tensor is just $0$, so the above proposition applies to cusp manifolds --- and more generally to any geometrically finite manifold with hyperbolic ends.

\begin{proof}
We denote $(p,u)$ for points of $T^\ast M$. If $X$ is a vector in $T_p M$, we denote by $X^h$ (resp. $X^v$) its horizontal (resp. vertical) lift, which are vectors in $T_{(p,u)}TM$.

Let $T$ be a vector valued tensor on $M$. From $T$ we can construct a variety of vector valued on $TM$. Indeed, first, we can construct tensors  on $M$ valued in $TTM$ by taking either the vertical of the horizontal lift of $T$. Then, we can compose $T$ by either $X^v \mapsto X$ or $X^h \mapsto X$. We consider now the class $\mathscr{B}_0$ of tensors on $TM$ that are obtained in this way when $T$ and all its derivatives are bounded. We also require that $0$-tensors $u^h$ and $u^v$ are in $\mathscr{B}_0$. Now, $\mathscr{B}$ is the smallest class of tensors stable by composition and sums that contains $\mathscr{B}_0$. 

From the formulae page 16 (prop. 7.5) for the curvature tensor of the Sasaki metric, we see that it is in $\mathscr{B}$ since the curvature tensor $R$ of $M$ as well as all its derivatives are bounded. The vector of the geodesic flow also is in $\mathscr{B}$ because it is $V(p,u) = u^h$.

We want to prove that $\mathscr{B}$ is \emph{stable under covariant derivatives}. We work in local coordinates. Observe that since covariant derivatives behave well with composition and sums, it suffices to prove that covariant derivatives of elements of $\mathscr{B}_0$ are in $\mathscr{B}$.

Let $p\in M$, let $U$ be some small open set containing $p$ where the normal coordinates at $p$, $\exp_p^{-1} : U \to T_p M$ are well defined. Taking an orthonormal basis $X_1, \dots, X_n$ in $T_p M$, we have coordinates $x_1, \dots, x_n$ on $U$. Then, we can consider coordinates $v_1, \dots, v_{2n}$ on $TU$ as in page 6 of \cite{MR1888866}. Since we have taken normal coordinates, the Christoffel coefficients vanish at $p$, and we have (see lemma 4.3 p. 7)
\begin{equation*}
\partial_{x_i}(p)^h = \partial_{v_i} \quad \partial_{x_i}(p)^v = \partial_{v_{n+i}}.
\end{equation*}
At a point $(p',u)$, we have 
\begin{equation}\label{eq:decomposition_local_basis}
u= \sum v_{n+k} \partial_{x_k}(p').
\end{equation}
Since we have taken normal coordinates, the $\nabla_{\partial_{x_1}} \partial_{x_i}$ vanish at $p$. From this and the formulae for covariant derivatives in proposition 7.2 page 15, we find
\begin{equation*}
\nabla_{a^h + b^v} u^h = b^h + \frac{1}{2} (R_p(u,b)u)^h - \frac{1}{2} (R_p(a,u)u)^v.
\end{equation*}
and 
\begin{equation*}
\nabla_{a^h + b^v} u^v = b^v + \frac{1}{2} (R_p(u,u)a)^h.
\end{equation*}

Now, we take $T$ a tensor on $M$ with all its derivatives bounded, and we just consider the case when $T$ is a $1$ tensor, and $T'(a^h + b^v) = (T(a))^h$. This defines an element of $\mathscr{B}_0$. 
\begin{equation*}
(\nabla_{X^h + Y^v} T')(a^h + b^v) = \nabla_{X^h + Y^v} (T(a))^h - T'( \nabla_{X^h + Y^v} (a^h + b^v)).
\end{equation*}
Using again the formulae for Sasaki covariant derivatives, we can expand this expression. There will be terms containing $\nabla_X T$ and terms involving $R_p$, $u$ and $T$, so the result \emph{will} be an element of $\mathscr{B}$.

To give a complete proof, we would have to consider all the possibilities that lead to similar computations; we leave this as an exercise for the reader.
\end{proof}

\subsection{The Sasaki metric in a cusp, and symbols}

Now, $M$ is a cusp manifold. The Sasaki metric is \emph{a priori} defined on the tangent space. However, there is a correspondance $v \mapsto \langle v, . \rangle$ between $T M$ and $T^\ast M$, and we define the Sasaki metric on $T^\ast M$ by pushing forward the metric on $T M$. As a consequence, $T^\ast M$ is endowed with a connection $\nabla$ and $\mathscr{C}^k$ norms. The following fact is the key to proving the Egorov lemma \ref{thm:Egorov_lemma}.
\begin{prop}\label{prop:equivalence_norms_Sasaki}
Take $E>0$, and consider functions on $T^\ast M$ supported in $(T^\ast M)_E:=\{ p(\xi) \leq E\}$. For such functions, the $\mathscr{C}^k(T^\ast M)$ norm is equivalent to the norm given by symbol estimates with $k' \leq k$ derivatives.
\end{prop}

\begin{proof}
The part of $(T^\ast M)_E$ above the compact part of $M$ is relatively compact, so all $C^k$ norms over it are equivalent. We just have to work in the cusps. Let us first start by finding the expression for the Sasaki metric in a cusp $Z$; we use again \cite{MR1888866}. We have coordinates $y, \theta$, and the coordinates on the tangent space $v_y, v_\theta$. From \eqref{eq:Christoffel_symbols}, we can compute
\begin{equation*}
\begin{split}
y\nabla_{\partial_y} \partial_y &= -\partial_y \\
y\nabla_{\partial_y} \partial_{\theta_i} &= - \partial_{\theta_i} \\
y\nabla_{\partial_{\theta_i}} \partial_y &= - \partial_{\theta_i}\\
y\nabla_{\partial_{\theta_i}} \partial_{\theta_j} &= \delta_{ij} \partial_y 
\end{split}
\end{equation*}
We deduce that the Sasaki metric on $TM$ is
\begin{equation*}
g = \frac{1}{y^2} \left( dy^2 + d\theta^2 + (dv_y + \frac{1}{y}(v_\theta .d\theta - v_y dy))^2 + ( dv_\theta - \frac{1}{y}(v_\theta dy + v_y d\theta))^2 \right)
\end{equation*}
Now, $v_y = y^2 Y$, and $v_\theta = y^2 J$, so this gives on $T^\ast M$
\begin{equation*}
g =  \frac{ dy^2 + d\theta^2}{y^2} + y^2 \left( (dY + \frac{1}{y}(J.d\theta - Ydy))^2 + ( dJ - \frac{1}{y}(J dy + Y d\theta))^2 \right)
\end{equation*}
Recall that $p= |\xi|^2/2$ is the symbol of $-h^2 \Delta/2$. We get that 
\begin{equation*}
g(X_y) = g(X_{\theta_i}) = 1 + 2p \quad \text{ and } \quad g(X_Y) = g(X_{J_i}) = 1.
\end{equation*}
and when $k\neq i$ --- $\langle .,. \rangle$ being the scalar product,
\begin{equation*}
\langle X_y, X_{\theta_i} \rangle = \langle X_{\theta_k}, X_{\theta_i} \rangle = \langle X_Y, X_{J_i} \rangle = \langle X_{J_k}, X_{J_i} \rangle = 0
\end{equation*}
\begin{equation*}
\langle X_y, X_Y\rangle = - yY, \quad \langle X_y, X_{J_i} \rangle = - yJ_i, \quad \langle X_{\theta_i}, X_Y \rangle = y J_i, \quad \langle X_{\theta_i}, X_{J_k} \rangle = - \delta_{ik} y Y.
\end{equation*}
If we use the Koszul formula \cite{Paulin} to determine the covariant derivatives of $X_{y,\theta, Y, J}$, we will find that they are of the type $a X_y + b X_{\theta} + c X_Y + d X_J$, where $a,b,c,d$ are elements of $S^1_{\mathcal{V}}$ --- defined in the paragraph after \eqref{eq:symbol_laplacian}. As a consequence, if $\alpha$ is a finite sequence of $\alpha_j \in \{ y, \theta_i, Y, J_k\}$ of length $k$, there are symbols $f_\beta \in S_\mathcal{V}$ for all sequences $\beta$ of the same type, of length $k'<k$, such that $f_\beta$ is of order $\leq k-k'$, and
\begin{equation*}
\hat{\nabla}^k_{X_{\alpha_1} \dots X_{\alpha_k}} = X_\alpha + \sum_{\beta} f_\beta X_\beta
\end{equation*}
From this we deduce that on $(T^\ast M)_E$, the norms 
\begin{equation*}
\Big\{\sum_{|\alpha| \leq k} q_{n,\alpha}\Big\}_n \text{ and }\sum_{|\alpha| \leq k} \sup_{T^\ast M}\|\hat{\nabla}^k_{X_\alpha}\|
\end{equation*} 
are equivalent. We are left to prove that the latter is equivalent to the $\mathscr{C}^k(T^\ast M)$ norm. It is \emph{a priori} bounded by it, so we need to prove a lower bound.

We have coordinates in each $T_\xi (T^\ast Z)$ given by
\begin{equation*}
T_\xi (T^\ast Z) \owns X = u_y X_y + u_\theta X_\theta + u_Y X_Y + u_J X_J
\end{equation*}
this defines a map $u_\xi :T_\xi T^\ast Z \to \R^{2d+2}$. Let us endow $\R^{2d+2}$ with the Euclidean metric. The equivalence to the $\mathscr{C}^k(T^\ast M)$ norm is assured if both $u_\xi$ and $u_\xi^{-1}$ are bounded independently of $\xi$ as long as $p(\xi) \leq E$. Let us compute:
\begin{equation*}
\| u_y X_y + u_\theta X_\theta + u_Y X_Y + u_J X_J \|^2 = u_y^2 + u_\theta^2 + (u_Y +yJ . u_\theta - yY u_y)^2 + (u_J -yJu_y - yY u_\theta)^2.
\end{equation*}
This is a bounded, positive quadratic form $q$ on $\R^{2d+2}$. To end the proof, we need to show that there is some $C>0$ such that $q>C .Id$ independently of $\xi$, when $p(\xi)\leq E$. However, since $q$ is bounded by $1+2p$, it suffices to prove that its determinant is bigger than some positive constant not depending on $\xi$. Some elementary computations show that the determinant is actually
\begin{equation*}
(1+y^2 J^2)^{d-1} \geq 1.
\end{equation*}
\end{proof}

\bibliographystyle{alpha}
\bibliography{biblio_article_Quantum_evolution_observables}

\end{document}